\documentclass[a4paper]{article}

\usepackage{graphicx}
\usepackage{color}
\usepackage{stmaryrd}
\usepackage[all]{xy}

\usepackage[numbers]{natbib}
\usepackage{natbibspacing}
\renewcommand{\bibfont}{\small}

\usepackage{tocloft}
\definecolor{mylinkcolor}{rgb}{0.05,0.05,0.4}
\usepackage[linktocpage]{hyperref}
\hypersetup{colorlinks,
linkcolor=mylinkcolor,
citecolor=mylinkcolor,
urlcolor=mylinkcolor}

\usepackage{latexsym}
\usepackage{amssymb,amsmath}
\usepackage{color}
\usepackage{stmaryrd}
\usepackage{graphicx}

\newcommand{\bref}[1]{(\ref{#1})}
\newcommand{\ucontents}[2]{\addcontentsline{toc}{#1}{\numberline{}{#2}}}
\newcommand{\demph}[1]{\textbf{\textup{#1}}}

\newcommand{\cat}[1]{\mathcal{#1}}
\newcommand{\fcat}[1]{\mathbf{#1}}
\newcommand{\scat}[1]{\mathbf{#1}}

\newcommand{\such}{:}

\newcommand{\epsln}{\varepsilon}
\newcommand{\iso}{\cong}
\newcommand{\eqv}{\simeq}
\newcommand{\of}{\mathbin{\circ}}
\newcommand{\sub}{\subseteq}

\newcommand{\from}{\colon}

\newcommand{\dvd}{\mathrel{\mid}}

\newcommand{\N}{\mathbb{N}}
\newcommand{\Z}{\mathbb{Z}}

\newcommand{\R}{\mathbb{R}}
\newcommand{\C}{\mathbb{C}}

\DeclareMathOperator{\Cl}{Cl}

\DeclareMathOperator{\Hom}{Hom}

\DeclareMathOperator{\im}{im}

\newcommand{\op}{\mathrm{op}}
\newcommand{\pr}{\mathrm{pr}}
\newcommand{\copr}{\mathrm{copr}}

\newcommand{\Ab}{\fcat{Ab}}

\newcommand{\CptMet}{\fcat{CptMet}}
\newcommand{\FDVect}{\fcat{FDVect}}

\newcommand{\FinSet}{\fcat{FinSet}}

\newcommand{\Met}{\fcat{Met}}

\newcommand{\Set}{\fcat{Set}}

\newcommand{\mg}[1]{\lvert#1\rvert}

\newcommand{\toby}[1]{\xrightarrow{#1}}

\newcommand{\incl}{\hookrightarrow}
\newcommand{\monic}{\rightarrowtail}
\newcommand{\epic}{\twoheadrightarrow}

\newcommand{\oppair}[4]{%
\xymatrix@1{%
#1 \ar@<.5ex>[r]^{#3} &{#2}\ar@<.5ex>[l]^{#4}%
}}

\newcommand{\hadjnli}[4]{%
\xymatrix@1{
#1 \ar@<1ex>[r]^-{#3} \ar@{}[r]|-\bot &#2 \ar@<1ex>[l]^-{#4}}}

\newcommand{\hadjnri}[4]{%
\xymatrix@1{
#1 \ar@<1ex>[r]^-{#3} \ar@{}[r]|-\top &#2 \ar@<1ex>[l]^-{#4}}}

\newcommand{\lengths}{\setlength{\unitlength}{1mm}\setlength{\fboxsep}{0pt}}
\newcommand{\cell}[4]{\put(#1,#2){\makebox(0,0)[#3]{#4}}}

\def\today{\number\day\space \ifcase\month\or
  January\or February\or March\or April\or May\or June\or
  July\or August\or September\or October\or November\or December\fi
  \space\number\year}

\newcommand{\ei}[1]{\im^\infty(#1)}
\newcommand{\eisole}{\im^\infty}
\newcommand{\ek}[1]{\ker^\infty(#1)}
\newcommand{\au}[1]{\widetilde{#1}}
\DeclareMathOperator{\coim}{coim}
\newcommand{\edi}[1]{\coim^\infty(#1)}
\newcommand{\ento}[2]{#1^{\rotatebox{110}{$\circlearrowright$}\!#2}}
\renewcommand{\C}{\cat{C}}
\newcommand{\Endo}[1]{\mathbf{En}(#1)}
\newcommand{\Auto}[1]{\mathbf{Au}(#1)}
\DeclareMathOperator{\colim}{colim}
\newcommand{\axstyle}[1]{\textbf{#1}}
\newcommand{\axI}{\axstyle{I}}
\newcommand{\axII}{\axstyle{II}}
\newcommand{\axIII}{\axstyle{III}}
\newcommand{\axIs}{\axstyle{I\/${}^*$}}
\newcommand{\axIIs}{\axstyle{II\/${}^*$}}
\newcommand{\axIIIs}{\axstyle{III\/${}^*$}}
\newcommand{\mapname}[1]{\ulcorner #1 \urcorner}
\newcommand{\emb}{\monic}
\newcommand{\cov}{\epic}
\newcommand{\embby}[1]{\stackrel{#1}{\rightarrowtail}}
\newcommand{\covby}[1]{\stackrel{#1}{\twoheadrightarrow}}
\newdir{ >}{{}*!/-10pt/@{>}} 
\newdir{> }{{}*!/15pt/@{>}}
\newcommand{\Emb}{\mathbf{Emb}}

\DeclareMathOperator{\can}{can}
\DeclareMathOperator{\spn}{span}
\newcommand{\metspn}[1]{\langle #1 \rangle}
\newcommand{\lefts}{\left(\!\!\!} 
\newcommand{\rights}{\!\!\!\right)}

\usepackage[amsmath,thmmarks]{ntheorem}

\newtheorem{thm}{Theorem}[section]
\newtheorem{propn}[thm]{Proposition}
\newtheorem{lemma}[thm]{Lemma}
\newtheorem{cor}[thm]{Corollary}

\theorembodyfont{\normalfont}

\newtheorem{defn}[thm]{Definition}
\newtheorem{example}[thm]{Example}
\newtheorem{examples}[thm]{Examples}
\newtheorem{remark}[thm]{Remark}

\theoremstyle{nonumberplain}
\theoremsymbol{\ensuremath{\square}}
\qedsymbol{\ensuremath{\square}}

\newtheorem{proof}{Proof}

\newcommand{\theoremtobeproved}{}
\newtheorem{pfoftheorem}{Proof of \theoremtobeproved}

\title{The eventual image}
\author{Tom Leinster%
\thanks{School of Mathematics, University of Edinburgh, Scotland;
  Tom.Leinster@ed.ac.uk}}
\date{\emph{In memory of Pieter Hofstra}}

\begin{document}

\sloppy
\maketitle

\begin{abstract}
In a category with enough limits and colimits, one can form the universal
automorphism on an endomorphism in two dual senses.  Sometimes these dual
constructions coincide, including in the categories of finite sets,
finite-dimensional vector spaces, and compact metric spaces. There,
beginning with an endomorphism $f$, there is a doubly-universal
automorphism on $f$ whose underlying object is the eventual image
$\bigcap_{n \geq 0} \im(f^n)$.  Our main theorem unifies these examples,
stating that in any category with a factorization system satisfying certain
axioms, the eventual image has two dual universal properties. A further
theorem characterizes the eventual image as a terminal coalgebra. In all,
nine characterizations of the eventual image are given, valid at different
levels of generality.
\end{abstract}

\tableofcontents\

\section{Introduction}
\label{sec:intro}

Any endomorphism in a suitably complete category $\C$ gives rise to an
automorphism in $\C$ in two dual universal ways. Indeed, let $\ento{X}{f}$
be an endomorphism in $\C$. There is an automorphism $\ento{L}{u}$ in
$\C$ together with a map $\ento{L}{u} \to \ento{X}{f}$ 
that is terminal among all maps from an
automorphism to $\ento{X}{f}$. Dually, there is another automorphism
$\ento{M}{v}$ with a map $\ento{X}{f} \to \ento{M}{v}$ that is initial as
such.

This much is a categorical triviality, following from the existence of
Kan extensions. Less trivial is the observation that in categories whose
objects are sufficiently finite in nature, these two dual universal
constructions coincide. For example, they coincide in the categories of
finite sets, of finite-dimensional vector spaces, and of compact metric
spaces.

Every endomorphism $\ento{X}{f}$ in such a category $\C$ therefore gives
rise to a single object equipped with an automorphism, with two dual
universal properties.  In the examples just mentioned, this object can be
constructed as $\bigcap_{n \in \N} \im(f^n)$, the \emph{eventual image} of
$f$. It can also be characterized as the space of points $x \in X$ that are
periodic in a category-sensitive sense: $x$ belongs to the set $\{f(x),
f^2(x), \ldots\}$ in the set case, or its span in the linear case, or its
closure in the metric case.

This work is intended as a small step towards a categorical treatment of
dynamical systems. An endomorphism $\ento{X}{f}$ can be seen as a discrete
time dynamical system in which $f$ is performed once with every tick of the
clock. The dynamical viewpoint is appropriate when $f$ is to be iterated
indefinitely. For example, Devaney's introductory dynamics text
(\cite{Deva}, p.~17) states:
\begin{quote}
The basic goal of the theory of dynamical systems is to understand the
eventual or asymptotic behaviour of an iterative process.
\end{quote}
Compared to the dynamical systems studied by practitioners of the subject,
the ones considered here are very primitive. But if there is to be any hope
of developing a helpful categorical approach to the theory of dynamical
systems in all its subtlety and complexity, we must first learn to handle
the most basic situations.

We begin with the definitions (Section~\ref{sec:defns}). Following the
principle that Kan extensions are best done pointwise, the eventual image
of $\ento{X}{f}$ is defined via (co)limits: if the diagram
\begin{align}
\label{eq:intro-seq}
\cdots \toby{f} X \toby{f} X \toby{f} X \toby{f} \cdots
\end{align}
has both a limit and a colimit, and the canonical map between them is an
isomorphism, we say that $f$ \emph{has eventual image duality} and define
the \emph{eventual image} $\ei{f}$ to be that (co)limit. We show that
$\ei{f}$ carries a canonical automorphism $\au{f}$ and that
$\ento{\ei{f}}{\au{f}}$ has the requisite universal properties. The $0$th
limit projection and colimit coprojection are maps
\[
\xymatrix@C+1em{
\ento{\ei{f}}{\au{f}} 
\ar@<0.5ex>[r]^-{\iota_f} &
\ento{X}{f}
\ar@<0.5ex>[l]^-{\pi_f}
}
\]
satisfying $\pi_f \iota_f = 1_{\ei{f}}$. Hence $\iota_f \pi_f$ is
an idempotent on $X$, called $f^\infty$, whose image is $\ei{f}$. Thus, the
original dynamical system $\ento{X}{f}$ gives rise not only to the
reversible system $\ento{\ei{f}}{\au{f}}$, but also to a system
$\ento{X}{f^\infty}$ that stabilizes in a single step.

With the definitions made, we prove some general results
(Section~\ref{sec:props}). For example, $\ei{f^n} = \ei{f}$ and
$(f^n)^\infty = f^\infty$ for every $n \geq 1$, meaning that the eventual
image construction is independent of timescale. We also relate eventual
images to shift equivalence, a standard relation in symbolic dynamics.

The main theorem (Section~\ref{sec:fact}) states that if $\C$ admits a
factorization system of `finite type' (defined there), then every
endomorphism in $\C$ has eventual image duality. The theorem not only
proves the existence of eventual images, but also provides two dual
explicit constructions. Indeed, $\ei{f}$ is constructed as $\bigcap_{n \in
\N} \im(f^n)$, or more precisely, as the limit of the diagram
\[
\cdots \emb \im(f^2) \emb \im(f) \emb X.
\]
It is also the colimit of the dual diagram. The main
theorem applies to all three categories mentioned: finite sets,
finite-dimensional vector spaces, and compact metric spaces with
distance-decreasing maps.

In the definition of factorization system of finite type, the main
condition is that an endomorphism belonging to either the left or the right
class must be invertible. This is a Dedekind finiteness condition, together
with its dual. Many categories of interest satisfy one condition or the
other~\cite{Uday}, but categories satisfying both are rarer.

In our three main examples, the eventual image of $\ento{X}{f}$ is the
largest subspace $A$ of $X$ satisfying $A \sub fA$. Generally, we prove
that in a category with a factorization system of finite type,
the eventual image is the terminal coalgebra for an endofunctor $A \mapsto
fA$ (Section~\ref{sec:coalg}).

Sections~\ref{sec:set}, \ref{sec:vect} and~\ref{sec:met} analyse the
eventual image in the categories of finite sets, finite-dimensional vector
spaces and compact metric spaces. In all three cases, we find explicit
descriptions of the idempotent $f^\infty$ on $X$. In the first, $f^\infty$
belongs to the set $\{1, f, f^2, \ldots\}$, in the second, it is in its
linear span, and in the third, it is in its closure. Finally,
Section~\ref{sec:other} gathers further examples of the eventual image. For
instance, in a Cauchy-complete category whose hom-sets are finite, every
endomorphism has eventual image duality.

In summary, we describe the eventual image of an endomorphism
$\ento{X}{f}$ in nine equivalent ways:
\begin{enumerate}
\item 
\label{item:char-univ}
as the universal automorphism equipped with a map into $\ento{X}{f}$;

\item
\label{item:char-lim}
as the limit of the diagram $\cdots \toby{f} X \toby{f} X \toby{f} \cdots$;

\item
\label{item:char-int}
as the limit of the diagram $\cdots \emb \im(f^2) \emb \im(f) \emb X$;

\item
\label{item:char-tc}
as the terminal coalgebra for the endofunctor $A \mapsto fA$ on subobjects
$A \emb X$;

\item
\label{item:char-per}
as the space of periodic points of $f$,
\end{enumerate}
together with the duals of
\bref{item:char-univ}--\bref{item:char-tc}. These descriptions are valid at
different levels of generality: \bref{item:char-univ}
and~\bref{item:char-lim} whenever $f$ has eventual image duality,
\bref{item:char-int} and~\bref{item:char-tc} when $\C$ has a factorization
system of finite type, and~\bref{item:char-per} for the three leading
examples, interpreting `periodic' appropriately.

\paragraph*{Related and further work} 

The eventual image appears in both symbolic dynamics and semigroup theory
(for instance, Definition~7.4.2 of~\cite{LiMa} and p.~79
of~\cite{SteiRTF}). There, it is more often called the eventual range,
although `eventual image' has been used (e.g.\ \cite{Kitc}, p.~53). The
idempotent $f^\infty$ is sometimes written as $f^\omega$, evoking the
countable ordinal $\omega$. While the eventual image $\ei{f} =
\im(f^\infty)$ is indeed the countable intersection $\bigcap_{n \in \N}
\im(f^n)$ in the cases studied here, that is only because of their finite
character. There are other settings in which the restriction of $f$ to
$\bigcap_{n \in \N} \im(f^n)$ need not be surjective, and one has to
iterate further through the ordinals to obtain an automorphism. Notation
aside, semigroup theory provides a lens through which to view this work
(Steinberg~\cite{EI,EI2}).

In topological dynamics, limits of diagrams of the
form~\eqref{eq:intro-seq} are known as generalized solenoids
(\cite{WillCOD}, p.~341).

Limit-colimit coincidences are closely related to absolute (co)limits:
consider direct sums in $\Ab$-categories and idempotent splittings, for
instance, or see Section~6 of Kelly and Schmitt~\cite{KeSc}. But some
simultaneous limits and colimits are not absolute. The eventual image is
one case (Example~\ref{eg:not-abs}); another can be found in the
representation theory of finite groups, where induced and coinduced
representations coincide.

Sections~\ref{sec:set}--\ref{sec:met} reveal many commonalities between the
three principal examples (sets, vector spaces and metric spaces). Some of
those commonalities are accounted for by the results of
Sections~\ref{sec:fact} and~\ref{sec:coalg} on factorization
systems. Others, such as the description of $\ei{f}$ in terms of periodic
points, are not. These merit further investigation.

One can also seek to generalize the examples given. The vector space case
can perhaps be extended to more general categories of modules. In the
metric case, our maps are distance-decreasing, but the crucial feature of
distance-decreasing endomorphisms $f$ is that $\{f^n \such n \in \N\}$ is
equicontinuous---a condition that refers only to the uniform structure, not
the metric (as noted by Steinberg~\cite{EI2}). 

Many of the results on sets, vector spaces and metric spaces in
Sections~\ref{sec:set}--\ref{sec:met} are elementary; they are not claimed to
be original. They are assembled in this way in order to reveal the 
common patterns and show their place in the theory of the eventual image.

Our focus is on the very special categories in which the left and right
universal methods for converting an endomorphism into an automorphism
coincide. We leave open many questions about more general categories. For
example, there is a sense in which the eventual image of the continuous map
$z \mapsto z^2$ on the Riemann sphere $\mathbb{C} \cup \{\infty\}$ ought to
be $\{z \in \mathbb{C} \such \mg{z} = 1\} \cup \{0, \infty\}$, even though
this is not what is given by any of the nine characterizations listed
above. Developing a general theory of the eventual image that covers such
examples remains a challenge.

\paragraph*{}

This paper is dedicated to the memory of Pieter
Hofstra, whose death is such a terrible loss.

\section{Definitions}
\label{sec:defns}

Let $\C$ be a category and let $f \from X \to X$ be an endomorphism in
$\C$, which we write as $\ento{X}{f}$. Suppose that the diagram
\begin{align}
\label{eq:dbl-seq}
\xymatrix{
\cdots \ar[r]^f &
X \ar[r]^f      &
X \ar[r]^f      &
X \ar[r]^f      &
\cdots
}
\end{align}
has both a limit cone $\bigl(L \toby{\pr_n} X\bigr)_{n \in \Z}$ and a
colimit cone $\bigl(X \toby{\copr_n} M\bigr)_{n \in \Z}$:
\[
\xymatrix@C+2em{
&
&
L \ar[ld]_{{\displaystyle\cdots} \qquad\qquad}|*+<2pt>{\scriptstyle\pr_{-1}} 
\ar[d]|*+<3pt>{\scriptstyle\pr_0} 
\ar[rd]|*+<3pt>{\scriptstyle\pr_1}^{\qquad\qquad{\displaystyle\cdots}}       &
&
\\
\cdots \ar[r]^f &
X \ar[r]^f \ar[rd]_{{\displaystyle\cdots}\qquad\qquad}|*+<3pt>{\scriptstyle\copr_{-1}}        &
X \ar[r]^f \ar[d]|*+<3pt>{\scriptstyle\copr_0}     &
X \ar[r]^f \ar[ld]|*+<3pt>{\scriptstyle\copr_1}^{\qquad\qquad{\displaystyle\cdots}}   &
\cdots          \\
&
&
M
&
&
}
\]
Then there is a canonical map $L \to M$, defined as the composite
\[
L \toby{\pr_n} X \toby{\copr_n} M,
\]
which is independent of $n \in \Z$. 

\begin{defn}
\label{defn:ei}
An endomorphism $\ento{X}{f}$ in a category \demph{has eventual image
duality} if the diagram~\eqref{eq:dbl-seq} has both a limit $L$ and a
colimit $M$, and the canonical map $L \to M$ is an isomorphism. In that
case, the object $L \iso M$, together with the limit projections and
colimit coprojections, is an \demph{eventual image} of $f$.
\end{defn}

\begin{example}
\label{eg:evim-set}
In $\Set$, most endomorphisms $\ento{X}{f}$ do not have eventual image
duality. 

The limit $L$ of~\eqref{eq:dbl-seq} is the set of all double sequences
$(x_n)_{n \in \Z}$ such that $x_n \in X$ and $f(x_n) = x_{n + 1}$ for all
$n \in \Z$. 

The colimit $M$ is the set of equivalence classes of pairs $(n, x)$ with $n
\in \Z$ and $x \in X$, where $(n, x) \sim (m, y)$ if there is some $p \geq
m, n$ such that $f^{p - n}(x) = f^{p - m}(y)$. Alternatively, it is the set
of equivalence classes of tails $(x_n)_{n \geq N}$ with $N \in \Z$ and
$f(x_n) = x_{n + 1}$ for all $n \geq N$, where two tails $(x_n)_{n \geq N}$
and $(y_n)_{n \geq P}$ are equivalent if $x_n = y_n$ for all sufficiently
large $n$.

The canonical map $L \to M$ sends $(x_n)_{n \in \Z}$ to the equivalence
class of $(0, x_0)$ in the first description of the colimit, or the
equivalence class of $(x_n)_{n \geq 0}$ in the second. It is typically
not bijective. For example, it is not surjective when $f$ is the squaring
map on the real interval $[2, \infty)$, and not injective when $f$ is the
squaring map on $\mathbb{C}$.
\end{example}

When $f$ has eventual image duality, we identify the limit and colimit
objects via the canonical isomorphism, writing $\ei{f}$ for both. Thus, we
have limit and colimit cones
\[
\bigl( \ei{f} \toby{\pr_n} X \bigr)_{n \in \Z},
\qquad
\bigl( X \toby{\copr_n} \ei{f} \bigr)_{n \in \Z}.
\]
Write
\[
\iota_f = \pr_0,
\qquad
\pi_f = \copr_0.
\]
Then the composite
\[
\ei{f} \toby{\iota_f} X \toby{\pi_f} \ei{f}
\]
is the identity, so there is an idempotent $\ento{X}{f^\infty}$ defined
by 
\[
f^\infty = \bigl( X \toby{\pi_f} \ei{f} \toby{\iota_f} X \bigr).
\]

The map of diagrams 
\begin{align*}
\begin{array}{c}
\xymatrix{
\cdots  \ar[r]^f        &
X \ar[r]^f \ar[d]^f     &
X \ar[r]^f \ar[d]^f     &
X \ar[r]^f \ar[d]^f     &
\cdots                  \\
\cdots \ar[r]_f         &
X \ar[r]_f              &
X \ar[r]_f              &
X \ar[r]_f              &
\cdots
}
\end{array}
\end{align*}
induces an endomorphism $\ento{\ei{f}}{\au{f}}$. In principle it induces
two such endomorphisms, depending on whether $\ei{f}$ is viewed as a limit
or a colimit, but they are equal.

\begin{lemma} 
Let $\ento{X}{f}$ be an endomorphism with eventual image duality, in any
category. Then $\ento{\ei{f}}{\au{f}}$ is invertible. 
\end{lemma}

\begin{proof}
There is a cone $\bigl( \ei{f} \toby{\pr^-_n} X \bigr)_{n \in \Z}$
on~\eqref{eq:dbl-seq} in which $\pr^-_n = \pr_{n - 1}$. Since $\bigl(
\ei{f} \toby{\pr_n} X \bigr)_{n \in \Z}$ is a limit cone, the cone
$(\pr^-_n)$ induces an endomorphism $\ento{\ei{f}}{}$, which one can
check is inverse to $\au{f}$.
\end{proof}

\begin{example}
\label{eg:ei-auto}
Every automorphism $\ento{X}{f}$ has eventual image duality. The eventual
image of $f$ is $X$ itself, with $\iota_f = \pi_f = f^\infty = 1_X$ and
$\au{f} = f$.
\end{example}

\begin{example}
\label{eg:ei-idem}
Every split idempotent has eventual image duality. Indeed, 
let $e\from X \to X$ be an idempotent in $\C$ splitting as 
\[
\xymatrix{
X \ar@<.3em>[d]^p    \\
I. \ar@<.3em>[u]^i
}
\]
Then $I$ is both the limit and colimit of $\cdots \toby{e} X \toby{e}
\cdots$, with projections $i$ and coprojections $p$. So $\ei{e} = I$ with
$\iota_e = i$, $\pi_e = p$, $e^\infty = e$ and $\au{e} = 1_I$.
\end{example}

Write $\Endo{\C}$ for the category whose objects $\ento{X}{f}$ are the
endomorphisms in $\C$ and whose maps $\ento{X}{f} \to \ento{Y}{g}$ are the
maps $u \from X \to Y$ in $\C$ satisfying $uf = gu$. It has a full
subcategory $\Auto{\C}$ consisting of the objects $\ento{X}{f}$ where $f$
is an automorphism.

Let $\ento{X}{f} \in \Endo{\C}$. If $f$ has eventual image duality then we
obtain an object $\ento{\ei{f}}{\au{f}} \in \Auto{\C}$ together with maps
\begin{align}
\label{eq:data}
\begin{array}{c}
\xymatrix{
\ento{X}{f} \ar@<.5em>[d]^{\pi_f}    \\
\ento{\ei{f}}{\au{f}} \ar@<.5em>[u]^{\iota_f}
}
\end{array}
\end{align}
such that $\pi_f \iota_f = 1$. Given only the idempotent $f^\infty$ in
$\C$, we can reconstruct $\ei{f}$, $\iota_f$ and $\pi_f$ as the splitting
data of $f^\infty$, and $\au{f}$ as $\pi_f \of f \of \iota_f$. In turn, the
data~\eqref{eq:data} determines the limit and colimit cones of
Definition~\ref{defn:ei} as follows.

\begin{lemma}
\label{lemma:pr}
Let $\C$ be a category. Let $\ento{X}{f}$ be an endomorphism in $\C$ with
eventual image duality. With notation as above, the limit and colimit cones
\[
\bigl( \ei{f} \toby{\pr_n} X \bigr)_{n \in \Z}, 
\qquad
\bigl( X \toby{\copr_n} \ei{f} \bigr)_{n \in \Z}
\]
are given by
\begin{align*}
\pr_n   &
= \Bigl( \ei{f} \toby{\au{f}^{\,n}} \ei{f} \toby{\iota_f} X \Bigr), \\
\copr_n &
= \Bigl( X \toby{\pi_f} \ei{f} \toby{\au{f}^{\,n}} \ei{f} \Bigr).
\end{align*}
\end{lemma}

\begin{proof}
By duality, it suffices to prove the first statement. Recall that
$\iota_f = \pr_0$. When $n \geq 0$,
the diagram
\[
\xymatrix{
\ei{f} \ar[r]^-{\iota_f} \ar[d]_{\au{f}^{\,n}} 
\ar@/^4pc/[rd]^{\pr_n}                         &
X \ar[d]_{f^n}                                  \\
\ei{f} \ar[r]_-{\iota_f} &
X
}
\]
commutes, the square by definition of $\au{f}$ and the triangle by the cone
property of $(\pr_n)$. Similarly, for $n \geq 0$, the diagram
\[
\xymatrix{
\ei{f} \ar[r]^-{\pr_{-n}} \ar[d]_{\au{f}^{n}} 
\ar@/^4pc/[rd]^{\pr_0 = \iota_f}                &
X \ar[d]_{f^{n}}                               \\
\ei{f} \ar[r]_-{\pr_{-n}}  &
X
}
\]
commutes, giving $\pr_{-n} = \iota_f \of \au{f}^{-n}$. 
\end{proof}

\begin{propn}
\label{propn:lim-univ}
Let $\C$ be a category and let $f\from X \to X$ be an endomorphism in $\C$
with eventual image duality. Then:
\begin{enumerate}
\item 
\label{part:lu-lim}
$\iota_f \from \ento{\ei{f}}{\au{f}} \to \ento{X}{f}$ is terminal among maps to
$\ento{X}{f}$ from objects of $\Auto{\C}$; 

\item
\label{part:lu-colim}
$\pi_f \from \ento{X}{f} \to \ento{\ei{f}}{\au{f}}$ is initial among maps
from $\ento{X}{f}$ to objects of $\Auto{\C}$.
\end{enumerate}
\end{propn}

\begin{proof}
By duality, it is enough to prove~\bref{part:lu-lim}. One can check
terminality directly. Alternatively, denote by $B\N$ the one-object
category corresponding to the additive monoid $\N$, and similarly
$B\Z$. Then
\[
\Endo{\C} \eqv \C^{B\N},
\qquad
\Auto{\C} \eqv \C^{B\Z},
\]
and the inclusion $\Auto{\C} \incl \Endo{\C}$ is induced by the inclusion
$B\N \incl B\Z$. The standard end or limit formula for Kan
extensions reduces, in this case, to the statement that the right Kan
extension of $\ento{X}{f} \from B\N \to \C$ along $B\N \incl B\Z$ is
$\ento{\ei{f}}{\au{f}}$, with canonical map $\iota_f$.
\end{proof}

The universal property of the eventual image construction makes it
functorial. Explicitly, let $u \from \ento{X}{f} \to \ento{Y}{g}$ be a map
in $\Endo{\C}$. Assuming that both $\ento{X}{f}$ and $\ento{Y}{g}$ have
eventual image duality, the map of diagrams
\begin{align*}
\xymatrix{
\cdots  \ar[r]^f        &
X \ar[r]^f \ar[d]^u     &
X \ar[r]^f \ar[d]^u     &
X \ar[r]^f \ar[d]^u     &
\cdots                  \\
\cdots \ar[r]_g         &
Y \ar[r]_g              &
Y \ar[r]_g              &
Y \ar[r]_g              &
\cdots
}
\end{align*}
induces a map $u_*\from \ei{f} \to \ei{g}$ on the limits or, equivalently,
the colimits. This construction is functorial where defined: $(1_X)_* =
1_{\ei{f}}$ and $(vu)_* = v_* u_*$.

\begin{example}
\label{eg:ind-self}
Let $\ento{X}{f}$ be an endomorphism with eventual image duality, and view
$f$ as a map $\ento{X}{f} \to \ento{X}{f}$ in $\Endo{\C}$. By definition, 
\[
f_* = \au{f} \from \ei{f} \to \ei{f}.
\]
\end{example}

\begin{lemma}
\label{lemma:induced}
Let $\ento{X}{f}$ and $\ento{Y}{g}$ be endomorphisms in $\C$ with eventual
image duality, and let $u \from \ento{X}{f} \to \ento{Y}{g}$ be a map in
$\Endo{\C}$. Then:
\begin{enumerate}
\item
\label{part:ind-aut}
$u_* \from \ei{f} \to \ei{g}$ is a map
\[
u_* \from \ento{\ei{f}}{\au{f}} \to \ento{\ei{g}}{\au{g}}
\]
in $\Auto{\C}$;

\item
\label{part:ind-idem}
$u\from X \to Y$ is a map
\[
u \from \ento{X}{f^\infty} \to \ento{Y}{g^\infty}
\]
in $\Endo{\C}$.
\end{enumerate}
\end{lemma}

\begin{proof}
For~\bref{part:ind-aut}, $uf = gu$, hence $u_* f_* = g_* u_*$ by
functoriality. But $f_* = \au{f}$ and $g_* = \au{g}$ by
Example~\ref{eg:ind-self}. 

For~\bref{part:ind-idem}, the diagram
\[
\xymatrix{
\ei{f} \ar@/^1pc/[rr]^{f^\infty}
\ar[r]_-{\iota_f} \ar[d]_{u_*}          &
X \ar[r]_-{\pi_f} \ar[d]^u              &
\ei{f} \ar[d]^{u_*}                     \\
\ei{g} \ar@/_1pc/[rr]_{g^\infty}
\ar[r]^-{\iota_g}                       &
Y \ar[r]^-{\pi_g}                       &
\ei{f}
}
\]
commutes by definition of $u_*$. 
\end{proof}

\begin{lemma}
\label{lemma:ind-idem}
Let $\ento{X}{f}$ be an endomorphism in $\C$ with eventual image
duality. Then: 
\begin{enumerate}
\item 
\label{part:ii-comm}
the endomorphisms $f$ and $f^\infty$ commute;

\item
\label{part:ii-map}
$f^\infty \from X \to X$ is a map $\ento{X}{f} \to \ento{X}{f}$ in $\Endo{\C}$;

\item
\label{part:ii-ind}
the induced map $(f^\infty)_* \from \ei{f} \to \ei{f}$ is the identity.
\end{enumerate}
\end{lemma}

\begin{proof}
Part~\bref{part:ii-comm} is the case $g = u = f$ of
Lemma~\ref{lemma:induced}\bref{part:ind-idem}, and part~\bref{part:ii-map}
follows. For~\bref{part:ii-ind}, it is enough to show that for each $n \in
\Z$, the outer triangle of
\[
\xymatrix@C+2ex{
&
&
X \ar[dd]^{f^\infty}    \\
\ei{f} \ar@/^/[rru]^{\pr_n} 
\ar[r]|*+<2pt>{\scriptstyle\au{f}^{\,n}} \ar@/_/[rrd]_{\pr_n}  &
\ei{f} \ar[ru]_{\iota_f} \ar[rd]^{\iota_f}      &
\\
&
&
X
}
\]
commutes. The upper and lower triangles commute by Lemma~\ref{lemma:pr},
and the right-hand triangle commutes because $f^\infty = \iota_f \pi_f$
and $\pi_f \iota_f = 1$. 
\end{proof}

Now suppose that $\C$ \demph{has eventual image duality}, meaning that
every endomorphism in $\C$ does. There is a functor
\[
\eisole \from \Endo{\C} \to \Auto{\C}
\]
defined on objects by $\ento{X}{f} \mapsto \ento{\ei{f}}{\au{f}}$ and on
maps by $u \mapsto u_*$ (which is valid by
Lemma~\ref{lemma:induced}\bref{part:ind-aut}). There is also an inclusion
functor
\[
U \from \Auto{\C} \to \Endo{\C},
\]
with $\eisole \of U \iso 1_{\Auto{\C}}$ by Example~\ref{eg:ei-auto}, and
there are natural transformations
\[
\iota \from U \of \eisole \to 1_{\Endo{\C}},
\qquad
\pi \from 1_{\Endo{\C}} \to U \of \eisole
\]
whose components at $\ento{X}{f} \in \Endo{\C}$ are $\iota_f$ and $\pi_f$. They
satisfy $\pi \iota = 1$. Proposition~\ref{propn:lim-univ} implies:

\begin{propn}
Let $\C$ be a category with eventual image duality. Then the functor
$\eisole$ is both left and right adjoint to the inclusion $U \from
\Auto{\C} \to \Endo{\C}$. The units and counits of the adjunctions are
$\iota$, $\pi$ and the canonical isomorphism $\eisole \of U \iso
1_{\Auto{\C}}$, 
and the unit-counit composite
\[
U \of \eisole \toby{\iota} 1_{\Endo{\C}} \toby{\pi} U \of \eisole
\]
is the identity.
\qed
\end{propn}

Simultaneous left and right adjunctions are sometimes called ambidextrous
adjunctions or ambijunctions~\cite{Laud}.

\section{General properties of the eventual image}
\label{sec:props}

Here we establish some properties of the constructions $f \mapsto \ei{f}$
and $f \mapsto f^\infty$. They largely concern invariance: when do two
endomorphisms have the same eventual image or the same associated
idempotent?

Throughout this section, let $\C$ be a category. 

Despite the notation, $f^\infty \of f \neq f^\infty$ in general. For
example, they are not equal when $f$ is a nontrivial automorphism, by
Example~\ref{eg:ei-auto}. But it is true that $f^\infty \of f = f \of
f^\infty$, by Lemma~\ref{lemma:ind-idem}\bref{part:ii-comm}. Moreover:

\begin{propn}
\label{propn:timescale}
Let $\ento{X}{f}$ be an endomorphism in $\C$ with eventual image
duality. Let $n \geq 1$. Then $\ei{f^n} \iso \ei{f}$ and $(f^n)^\infty =
f^\infty$.
\end{propn}

\begin{proof}
The inclusion $(n\Z, \leq) \incl (\Z, \leq)$ is cofinal,%
\footnote{We use the terminological convention in which \emph{cofinal} functors
leave \emph{limits} unchanged.}
 so the
canonical map
\[
\ei{f}
=
\lim\Bigl( \cdots \toby{f} X \toby{f} \cdots \Bigr)
\ \toby{\ \ k\ \ }\ 
\lim\Bigl( \cdots \toby{f^n} X \toby{f^n} \cdots \Bigr)
=
\ei{f^n}
\]
is an isomorphism. A dual statement holds for colimits, giving an
isomorphism $\ell \from \ei{f^n} \to \ei{f}$. The diagram
\[
\xymatrix{
\ei{f} \ar[dd]_k \ar[rd]^{\iota_f} \ar[rr]^1    &
&
\ei{f}  \\ 
&
X \ar[ru]^{\pi_f} \ar[rd]_{\pi_{f^n}}           &
\\
\ei{f^n} \ar[ru]_{\iota_{f^n}} \ar[rr]_1        &
&
\ei{f^n} \ar[uu]_\ell
}
\]
commutes, and $k$ and $\ell$ are isomorphisms, so $\ell = k^{-1}$. But now
the commutative diagram
\[
\xymatrix{
&
\ei{f} \ar[rd]_{\iota_f} \ar[dd]^k      &
\\
X 
\ar@/^5pc/[rr]^{f^\infty} \ar@/_5pc/[rr]_{(f^n)^\infty}
\ar[ru]_{\pi_f} \ar[rd]^{\pi_{f^n}}     &
&
X       \\
&
\ei{f^n} \ar[ru]^{\iota_{f^n}}  &
\\
}
\]
shows that $f^\infty = (f^n)^\infty$. 
\end{proof}

The property of the eventual image established in
Proposition~\ref{propn:timescale} is shared by other dynamical
constructs. For example, every holomorphic self-map $f$ of a compact
Riemann surface $X$ has a Julia set $J(f) \sub X$, and $J(f^n) = J(f)$ for
all $n \geq 1$ (Lemma~4.2 of Milnor~\cite{Miln}). If our endomorphism $f$
is applied to $X$ once per second, then the equality $\ei{f} = \ei{f^{60}}$
means that the eventual image is the same whether the process is observed
every second or every minute: it is independent of timescale.

In symbolic and topological dynamics, there is a standard notion of shift
equivalence (Wagoner~\cite{Wago}; Williams~\cite{WillCOD}, p.~342).  Two
endomorphisms $\ento{X}{f}$ and $\ento{Y}{g}$ in $\C$ are \demph{shift
equivalent} if there exist $n \in \N$ and maps
\begin{align}
\label{eq:shift-eq}
\oppair{\ento{X}{f}}{\ento{Y}{g}}{u}{v}
\end{align}
in $\Endo{\C}$ such that $vu = f^n$ and $uv = g^n$. (Then the same is true
for all $N \geq n$: replace $u$ by $uf^{N - n}$.) When $\ento{X}{f}$ and
$\ento{Y}{g}$ both have eventual image duality, call $\ento{X}{f}$ and
$\ento{Y}{g}$ \demph{eventually equivalent} if there exist
maps~\eqref{eq:shift-eq} such that $vu = f^\infty$ and $uv = g^\infty$.

\begin{propn}
\label{propn:shift-eqv}
Let $\ento{X}{f}$ and $\ento{Y}{g}$ be endomorphisms in $\C$, both with eventual
image duality. Then
\begin{align*}
        &\ento{X}{f} \text{ and } \ento{Y}{g} \text{ are shift equivalent}
\\
\implies        
        &\ento{\ei{f}}{\au{f}} \iso \ento{\ei{g}}{\au{g}}      
\\
\iff    &\ento{X}{f} \text{ and } \ento{Y}{g} \text{ are eventually equivalent.}
\end{align*}
\end{propn}

\begin{proof}
Suppose that $\ento{X}{f}$ and $\ento{Y}{g}$ are shift equivalent, with
maps $u$ and $v$ as in~\eqref{eq:shift-eq}. They induce maps
\begin{align}
\label{eq:induced-op}
\oppair{\ento{\ei{f}}{\au{f}}}{\ento{\ei{g}}{\au{g}}}{u_*}{v_*},
\end{align}
which satisfy
\[
v_* u_* = (vu)_* = (f^n)_* = (f_*)^n = \au{f}^{\,n}
\]
(the last step by Example~\ref{eg:ind-self}). Hence $v_* u_*$ is an
isomorphism, and dually, so is $u_* v_*$. It follows that $u_*$ is an
isomorphism, giving $\ento{\ei{f}}{\au{f}} \iso \ento{\ei{g}}{\au{g}}$.

Next suppose that $\ento{\ei{f}}{\au{f}} \iso
\ento{\ei{g}}{\au{g}}$. Choose an isomorphism $k$, and define
\begin{align*}
u       &
= \Bigl( 
\ento{X}{f} \toby{\pi_f} 
\ento{\ei{f}}{\au{f}} \toby{k} 
\ento{\ei{g}}{\au{g}} \toby{\iota_g} \ento{Y}{g} \Bigr),  \\
v       &
= \Bigl(
\ento{Y}{g} \toby{\pi_g}
\ento{\ei{g}}{\au{g}} \toby{k^{-1}}
\ento{\ei{f}}{\au{f}} \toby{\iota_f}
\ento{X}{f} \Bigr).
\end{align*}
Then $vu = f^\infty$ and $uv = g^\infty$, so $\ento{X}{f}$ and
$\ento{Y}{g}$ are eventually equivalent.

Finally, suppose that $\ento{X}{f}$ and $\ento{Y}{g}$ are eventually
equivalent, and take maps $u$ and $v$ as in the definition. The induced
maps~\eqref{eq:induced-op} satisfy
\[
v_* u_* = (vu)_* = (f^\infty)_* = 1_{\ei{f}}
\]
by Lemma~\ref{lemma:ind-idem}\bref{part:ii-ind}, and dually, $u_* v_* =
1_{\ei{g}}$. Hence $u_*$ and $v_*$ are mutually inverse maps between
$\ento{\ei{f}}{\au{f}}$ and $\ento{\ei{g}}{\au{g}}$.
\end{proof}

\begin{cor}
\label{cor:vu-uv}
Let $\oppair{X}{Y}{u}{v}$ be maps in $\C$, and suppose that $vu$ and $uv$
have eventual image duality. Then $\ento{\ei{vu}}{\au{vu}} \iso
\ento{\ei{uv}}{\au{uv}}$.
\end{cor}

\begin{proof}
The maps $\oppair{\ento{X}{vu}}{\ento{Y}{uv}}{u}{v}$ in $\Endo{\C}$
define a shift equivalence, so Proposition~\ref{propn:shift-eqv} applies.
\end{proof}

Finally, let $f$ and $g$ be endomorphisms of the same object of $\C$. In
general, $(gf)^\infty \neq g^\infty f^\infty$. For example, let $f$ and $g$
be split idempotents such that $gf$ is not idempotent (such as the linear
operators on $\R^2$ represented by $\bigl(\begin{smallmatrix}0 &0 \\ 0&
1\end{smallmatrix}\bigr)$ and $\bigl(\begin{smallmatrix}1 &1\\ 0&
0\end{smallmatrix}\bigr)$).  By Example~\ref{eg:ei-idem}, $g^\infty
f^\infty$ is $gf$, which is not idempotent and so not equal to
$(gf)^\infty$. However:

\begin{propn}
Let $f$ and $g$ be commuting endomorphisms of an object $X$ of
$\C$. Suppose that $f$, $g$ and $gf$ have eventual image duality. Then
$(gf)^\infty = g^\infty f^\infty$. 
\end{propn}

\begin{proof}
Consider the commutative diagram
\[
\xymatrix{
\ddots \ar@{.>}[rd]^{gf}&
\vdots \ar[d]^g         &
\vdots \ar[d]^g         &
\\
\cdots \ar[r]^f         &
X \ar[r]^f \ar[d]^g \ar@{.>}[rd]^{gf}   &
X \ar[r]^f \ar[d]^g     &
\cdots  \\
\cdots \ar[r]^f         &
X \ar[r]^f \ar[d]^-g    &
X \ar[r]^f \ar[d]^-g \ar@{.> }[rd]^(0.4){gf}&
\cdots  \\
&
\vdots  &
\vdots  &
\ddots
}
\]
The solid part shows a functor $(\Z, \leq) \times (\Z, \leq) \to \C$, which
restricted to the diagonal is the functor $(\Z, \leq) \to \C$ shown as the
dotted part. Since the diagonal subset of $\Z \times \Z$ is cofinal, the
dotted and solid parts have the same limits. Now the dotted part has limit
$\ei{gf}$, and the limit of the solid part can be calculated by taking
limits in rows and then columns:
\[
\xymatrix{
\vdots \ar[d]_-{g_*}    &
&
\vdots \ar[d]^g         &
\vdots \ar[d]^g         &
\\
\ei{f} \ar[d]_-{g_*}    &
\cdots \ar[r]^f         &
X \ar[r]^f \ar[d]^g     &
X \ar[r]^f \ar[d]^g     &
\cdots  \\
\ei{f} \ar[d]_-{g_*}    &
\cdots \ar[r]^f         &
X \ar[r]^f \ar[d]^-g    &
X \ar[r]^f \ar[d]^-g    &
\cdots  \\
\vdots  &
&
\vdots  &
\vdots  
}
\]
Thus, the limit of $\cdots \toby{g_*} \ei{f} \toby{g_*} \cdots$ is
$\ei{gf}$, and the $0$th limit projection $\iota_{g_*}\from \ei{gf} \to
\ei{f}$ makes the triangle
\begin{align}
\label{eq:two-ways}
\begin{array}{c}
\xymatrix{
\ei{gf} \ar[rd]^{\iota_{gf}} \ar[d]_{\iota_{g_*}}       &
\\
\ei{f} \ar[r]_-{\iota_f} &
X
}
\end{array}
\end{align}
commute. The dual argument applies to colimits. 
Putting together
triangle~\eqref{eq:two-ways} with its dual gives a commutative diagram
\[
\xymatrix{
\ei{gf} \ar[d]_{\iota_{g_*}} \ar[rd]^{\iota_{gf}} 
&
&
\\
\ei{f} \ar[r]_-{\iota_f} 
&
X \ar[r]^-{\pi_f} \ar[rd]_{\pi_{gf}}     
&
\ei{f} \ar[d]^{\pi_{g_*}}       \\
&
&
\ei{gf}.
}
\]
Since $\pi_f \iota_f = 1$ and $\pi_{gf} \iota_{gf} = 1$, it follows that
$\pi_{g_*} \iota_{g_*} = 1$, so that $\ento{\ei{f}}{g_*}$ has eventual image
duality with $\ei{g_*} \iso \ei{gf}$. 

Now consider the diagram
\[
\xymatrix{
X \ar[r]^-{\pi_f} \ar[rd]_{\pi_{gf}}^{\text{\ (A)}}&
\ei{f} \ar[r]^-{\iota_f} \ar[d]^{\pi_{g_*}}      &
X \ar[d]^{\pi_g}        \\
&
\ei{gf} \ar[rd]^{\iota_{gf}}_{\text{(B)\ \ \ }}
\ar[d]_{\iota_{g_*}} \ar[r]^{(\iota_f)_*}       &
\ei{g} \ar[d]^{\iota_g} \\
&
\ei{f} \ar[r]_-{\iota_f} &
X.
}
\]
Triangles~(A) and~(B) have already been shown to commute, and the two
squares commute because $\iota_f$ is a map $\ento{\ei{f}}{g_*} \to
\ento{X}{g}$. Hence the triangle between the three copies of $X$ commutes;
that is, $(gf)^\infty = g^\infty f^\infty$.
\end{proof}

\section{Factorization systems and the main theorem}
\label{sec:fact}

Here we prove our main theorem: a category admitting a factorization
system of a suitable kind has eventual image duality.

Recall that a \demph{factorization system} on a category $\C$ consists of
subcategories $\cat{L}$ and $\cat{R}$, each containing all the objects and
isomorphisms, such that every map in $\C$ factorizes as a map in $\cat{L}$
followed by a map in $\cat{R}$ uniquely up to unique isomorphism
(\cite{FrKe}, Section~2). We call maps in $\cat{L}$ \demph{coverings} and
denote them by $\cov$; maps in $\cat{R}$ are \demph{embeddings},
$\emb$. The uniqueness of factorization up to unique isomorphism means that
for any solid commutative square
\[
\xymatrix@R-4ex{
&
I \ar@{ >->}[rd] \ar@{.>}[dd]_{\iso}^k   &       
\\
X \ar@{->>}[ru] \ar@{->>}[rd]   &
&
Y,      \\
&
I' \ar@{ >->}[ru]        &
}
\]
there is a unique isomorphism $k$ such that the triangles commute. When a
map $f\from X \to Y$ factorizes as $X \cov I \emb Y$, we write $I$ as
$\im(f)$. Typically we leave the maps $X \cov \im(f) \emb Y$ nameless.

The axioms have some standard elementary consequences (proofs omitted).

\begin{lemma}
\label{lemma:fs-basic}
Every factorization system has the following properties.
\begin{enumerate}
\item 
\label{part:fsb-iso}
\demph{(Isomorphisms)} A map that is both a covering and an embedding is an
isomorphism.

\item
\demph{(Two out of three)} 
For composable maps $f$ and $g$, if $gf$ and $g$ are embeddings then so is
$f$, and if $gf$ and $f$ are coverings then so is $g$;

\item
\label{part:fsb-orth}
\demph{(Orthogonality)} 
The coverings are left orthogonal to the embeddings: for any solid
commutative square
\[
\xymatrix{
W \ar[r] \ar@{->>}[d]   &
X \ar@{ >->}[d]          \\
Y \ar[r] \ar@{.>}[ru]   &
Z,
}
\]
there is a unique dotted arrow such that the triangles commute.

\item
\demph{(Functoriality)}
For any commutative square
\[
\xymatrix{
X \ar[r]^f \ar[d]_u     &
Y \ar[d]^w      \\
X' \ar[r]_{f'}  &
Y',
}
\]
there is a unique map $v\from \im(f) \to \im(f')$ such that
\[
\xymatrix{
X \ar@{->>}[r] \ar[d]_u &
\im(f) \ar@{ >->}[r] \ar@{.>}[d]|*+<4pt>{\scriptstyle v}       &
Y \ar[d]^w      \\
X' \ar@{->>}[r] &
\im(f') \ar@{ >->}[r]    &
Y'
}
\]
commutes.
\end{enumerate}
\qed
\end{lemma}

Let $\ento{X}{f}$ be an endomorphism in a category with a factorization
system. Let $n, k \geq 0$. Functoriality applied to the squares
\[
\xymatrix{
X \ar[r]^{f^{n + k}} \ar[d]_{f^k}       &
X \ar[d]^1      \\
X \ar[r]_{f^n}  &
X
}
\qquad 
\xymatrix{
X \ar[r]^{f^n} \ar[d]_1 &
X \ar[d]^{f^k}  \\
X \ar[r]_{f^{n + k}}    &
X
}
\]
gives unique dotted maps such that the diagrams 
\begin{align}
\label{eq:in-out}
\begin{array}{c}
\xymatrix{
X \ar@{->>}[r] \ar[d]_{f^k} &
\im(f^{n + k}) \ar@{ >->}[r] \ar@{.>}[d] &
X \ar[d]^1      \\
X \ar@{->>}[r] &
\im(f^n) \ar@{ >->}[r]    &
X
}
\end{array}
\qquad
\begin{array}{c}
\xymatrix{
X \ar@{->>}[r] \ar[d]_1 &
\im(f^n) \ar@{ >->}[r] \ar@{.>}[d]       &
X \ar[d]^{f^k}      \\
X \ar@{->>}[r] &
\im(f^{n + k}) \ar@{ >->}[r]    &
X
}
\end{array}
\end{align}
commute. By the two out of three property, the first dotted map is an
embedding and the second is a covering. We leave them nameless, writing
them as simply
\begin{align}
\label{eq:emb-cov}
\im(f^{n + k}) \emb \im(f^n),
\qquad
\im(f^n) \cov \im(f^{n + k}).
\end{align}
The uniqueness in~\eqref{eq:in-out} implies that $\im(f^n) \emb \im(f^n)$
is the identity and that
\[
\xymatrix@R-2ex{
&
\im(f^{n + k}) \ar@{ >->}[rd]    &
\\
\im(f^{n + k + \ell}) \ar@{ >->}[ru] \ar@{ >->}[rr]       &
&
\im(f^n)
}
\]
commutes for all $n, k, \ell \geq 0$, and dually. The embeddings and
coverings~\eqref{eq:emb-cov} are compatible in the following sense.

\begin{lemma}
\label{lemma:emb-cov}
Let $\ento{X}{f}$ be an endomorphism in a category with a factorization
system. Then for all $n, k, \ell \geq 0$, the square
\begin{align}
\label{eq:emb-cov-sq}
\begin{array}{c}
\xymatrix{
\im(f^{n + k}) \ar@{->>}[r] \ar@{ >->}[d]        &
\im(f^{n + k + \ell}) \ar@{ >->}[d]      \\
\im(f^n) \ar@{->>}[r]   &
\im(f^{n + \ell})
}
\end{array}
\end{align}
commutes.
\end{lemma}

\begin{proof}
By functoriality, there is a unique map $v$ such that that the diagram
\begin{align}
\label{eq:ec-1}
\begin{array}{c}
\xymatrix{
X \ar@{->>}[r] \ar@/^1pc/[rr]^{f^{n + k}} \ar[d]_{f^k}      &
\im(f^{n + k}) \ar@{ >->}[r] \ar@{.>}[d]|*+<4pt>{\scriptstyle v}       &
X \ar[d]^{f^\ell}       \\
X \ar@{->>}[r] \ar@/_1pc/[rr]_{f^{n + \ell}}        &
\im(f^{n + \ell}) \ar@{ >->}[r]  &
X
}
\end{array}
\end{align}
commutes. It is therefore enough to show that taking $v$ to be either
composite around the square~\eqref{eq:emb-cov-sq} makes~\eqref{eq:ec-1}
commute. That the clockwise composite does so follows from the
commutativity of the diagram
\[
\xymatrix{
X \ar@{->>}[r] \ar@/^1pc/[rr]^{f^{n + k}} \ar[d]_1 \ar@/_2pc/[dd]_{f^k} &
\im(f^{n + k}) \ar@{ >->}[r] \ar@{->>}[d]        &
X \ar[d]^{f^\ell} \ar@/^2pc/[dd]^{f^\ell}       \\
X \ar@{->>}[r] \ar[d]_{f^k}     &
\im(f^{n + k + \ell}) \ar@{ >->}[r] \ar@{ >->}[d] &
X \ar[d]^1      \\
X \ar@{->>}[r] \ar@/_1pc/[rr]_{f^{n + \ell}}    &
\im(f^{n + \ell}) \ar@{ >->}[r]  &
X,
}
\]
and a similar argument applies to the anticlockwise composite.
\end{proof}

We now formulate conditions on a factorization system expressing the idea
that the objects of the category are in some sense finite. The three main
examples are as follows; details can be found in Sections
\ref{sec:set}--\ref{sec:met}. 

\begin{examples}
\label{egs:fs-basic}
\begin{enumerate}
\item 
Let $\FinSet$ be the category of finite sets, with the factorization system
in which embeddings are injections and coverings are surjections.

\item
Let $\FDVect$ be the category of finite-dimensional vector spaces over a
field $k$, again with the injective and surjective maps as the embeddings
and coverings.

\item
Let $\CptMet$ be the category of compact metric spaces and
distance-decreasing ($1$-Lipschitz) maps. It has a factorization system in
which the embeddings are the distance-preserving maps and the coverings are
the surjective maps.
\end{enumerate}
\end{examples}

\begin{defn}
A factorization system is of \demph{finite type} if it satisfies the
following three axioms:
\begin{itemize}
\item[\axI] 
every endomorphism that is an embedding is an isomorphism;

\item[\axII] 
every sequence $\ \cdots \emb \cdot \emb \cdot$ has a limit;

\item[\axIII] 
for every commutative diagram
\[
\xymatrix{ 
\cdots  \ar@{ >->}[r]                     &
X_1 \ar@{ >->}[r] \ar@{->>}[d]        &
X_0 \ar@{->>}[d]        \\
\cdots  \ar@{ >->}[r]     &
Y_1 \ar@{ >->}[r]     &
Y_0, 
}
\]
the induced map $\lim X_n \to \lim Y_n$ is a covering;
\end{itemize}
together with their duals:
\begin{itemize}
\item[\axIs] 
every endomorphism that is a covering is an isomorphism;

\item[\axIIs] 
every sequence $\cdot \cov \cdot \cov \cdots\ $ has a colimit;

\item[\axIIIs] 
for every commutative diagram
\[
\xymatrix{ 
Y_0 \ar@{->>}[r] \ar@{ >->}[d]           &
Y_1 \ar@{->>}[r] \ar@{ >->}[d]           &
\cdots                  \\
X_0 \ar@{->>}[r]        &
X_1 \ar@{->>}[r]        &
\cdots,                  
}
\]
the induced map $\colim Y_n \to \colim X_n$ is an embedding.
\end{itemize}
\end{defn}

All three of Examples~\ref{egs:fs-basic} are of finite type, as shown in
Sections~\ref{sec:set}--\ref{sec:met}.

\begin{remark}
If the coverings in axiom~\axIII\ are replaced by embeddings then the
induced map $\lim X_n \to \lim Y_n$ is automatically an embedding, by the
two out of three property and Lemma~\ref{lemma:proj-emb} below. Hence
axiom~\axIII\ is equivalent to the statement that factorizations are
preserved by sequential limits of embeddings.
\end{remark}

We set out some elementary consequences of the axioms.

\begin{lemma}
\label{lemma:split-mono-emb}
In a factorization system satisfying axioms~\axI\ and~\axIs:
\begin{enumerate}
\item
\label{part:sme-iso}
every split monic covering is an isomorphism, and every split epic
embedding is an isomorphism;

\item
\label{part:sme-main}
every split monic is an embedding, and every split epic is a covering.
\end{enumerate}
\end{lemma}

\begin{proof}
By duality, it suffices to prove the first statement in each
part. For~\bref{part:sme-iso}, let $i \from Y \cov X$ be a split monic
covering, so that $pi = 1_Y$ for some $p \from X \to Y$. By the two out of
three property, $p$ is also a covering. Hence $ip \from X \to X$ is a
covering, which by axiom~\axIs\ implies that $ip$ is an isomorphism. So $i$
is epic as well as split monic, and is therefore an isomorphism.

For~\bref{part:sme-main}, let $i \from Y \to X$ be a split monic. Factorize
$i$ as
\[
i = \bigl( Y \covby{q} \im(i) \emb X \bigr).
\]
Then $q$ is also split monic. By~\bref{part:sme-iso}, $q$ is an
isomorphism, so $i$ is an embedding.
\end{proof}

\begin{example}
In a category with a factorization system satisfying~\axI\ and~\axIs,
Lemma~\ref{lemma:split-mono-emb}\bref{part:sme-main} implies that the
splitting object of a split idempotent $e$ is $\im(e)$.  So when $f$ is
an endomorphism with eventual image duality, $\ei{f} = \im(f^\infty)$.
\end{example}

\begin{lemma}
\label{lemma:proj-emb}
Let $\C$ be a category with a factorization system satisfying
axioms~\axI\ and~\axIs. Let
\begin{align}
\label{eq:pe}
\cdots \embby{f_1} X_1 \embby{f_0} X_0
\end{align}
be a diagram in $\C$ with a limit cone $\bigl(L \toby{j_n} X_n)_{n \in
\N}$. Then $j_n$ is an embedding for all $n \in \N$. The dual statement
holds for sequential colimits of coverings. 
\end{lemma}

\begin{proof}
First we show that $j_0, j_1, \ldots$ all have the same image. 
For each $n \in \N$, factorize $j_n$ as 
\[
j_n =
\bigl( L \covby{q_n} \im(j_n) \embby{k_n} X_n \bigr).  
\]
By the cone property, 
\[
j_n = f_n j_{n + 1} = 
\bigl(\!\!
\xymatrix{
L \ar@{->>}[r]^-{q_{n + 1}} &
\im(j_{n + 1}) \ar@{ >->}[r]^-{k_{n + 1}} &
X_{n + 1} \ar@{ >->}[r]^-{f_n} &
X_n
}
\!\!\bigr).
\]
By uniqueness of factorizations, there is a unique isomorphism $\im(j_n)
\toby{\sim} \im(j_{n + 1})$ compatible with these two factorizations of
$j_n$. Put $I = \im(j_0)$ and $q = q_0 \from L \cov I$, and let $j'_n \from
I \emb X_n$ be the composite
\[
I = \im(j_0) \toby{\sim} \cdots \toby{\sim} \im(j_n) \embby{k_n} X_n.
\]
One easily checks that for all $n \in \N$, 
\[
j_n = \Bigl( L \covby{q} I \embby{j'_n} X_n \Bigr),
\qquad
j'_n = \Bigl( I \embby{j'_{n + 1}} X_{n + 1} \embby{f_n} X_n \Bigr).
\]
Hence $\bigl( I \embby{j'_n} X_n \bigr)_{n \in \N}$ is a cone on the
diagram~\eqref{eq:pe}, so there is a unique map $r \from I \to L$ such that
\[
j'_n = \bigl( I \toby{r} L \toby{j_n} X_n \bigr)
\]
for all $n \in \N$. For each $n \in \N$ we have $j_n r q = j'_n q = j_n$,
and it follows from the limit property of $L$ that $rq = 1_L$.  So $q$ is a
split monic covering, hence an isomorphism by
Lemma~\ref{lemma:split-mono-emb}\bref{part:sme-iso}. But $j_n = j'_n q$, so
$j_n$ is an embedding.
\end{proof}

We can now prove the main theorem. 

\begin{thm}
\label{thm:main}
A category admitting a factorization system of finite type has eventual
image duality. Moreover, in such a category: 
\begin{enumerate}
\item
\label{part:main-lim}
the eventual image $\ei{f}$ of an endomorphism $\ento{X}{f}$ is the limit
of the diagram
\[
\cdots \emb \im(f^2) \emb \im(f) \emb X,
\]
the map $\iota_f \from \ei{f} \to X$ is the $0$th projection of the limit
cone, and $\ento{\ei{f}}{\au{f}}$ is the map on limits induced by
the map of diagrams
\begin{align}
\label{eq:main-ladder}
\begin{array}{c}
\xymatrix{
\cdots \ar@{ >->}[r]     &
\im(f^2) \ar@{ >->}[r] \ar@{->>}[d]      &
\im(f)   \ar@{ >->}[r] \ar@{->>}[d]      &
X         \ar@{->>}[d]   &
\\
\cdots \ar@{ >->}[r]     &
\im(f^3) \ar@{ >->}[r]   &
\im(f^2) \ar@{ >->}[r]   &
\im(f) ;
}
\end{array}
\end{align}

\item
\label{part:main-colim}
dually,
the eventual image $\ei{f}$ of $\ento{X}{f}$ is the colimit of the diagram
\[
X \cov \im(f) \cov \im(f^2) \cov \cdots,
\]
the map $\pi_f \from X \to \ei{f}$ is the $0$th coprojection of the
colimit cone, and $\ento{\ei{f}}{\au{f}}$ is the map on colimits
induced by the map of diagrams
\begin{align}
\label{eq:dual-ladder}
\begin{array}{c}
\xymatrix{
\im(f) \ar@{->>}[r] \ar@{ >->}[d]       &
\im(f^2) \ar@{->>}[r] \ar@{ >->}[d]     &
\im(f^3) \ar@{->>}[r] \ar@{ >->}[d]     &
\cdots  \\
X        \ar@{->>}[r]   &
\im(f)   \ar@{->>}[r]   &
\im(f^2) \ar@{->>}[r]   &
\cdots  .
}
\end{array}
\end{align}
\end{enumerate}
\end{thm}

The commutativity of diagrams~\eqref{eq:main-ladder}
and~\eqref{eq:dual-ladder} follows from Lemma~\ref{lemma:emb-cov}.

\begin{proof}
Let $\C$ be a category with a factorization system of finite type, and let
$\ento{X}{f}$ in $\C$. The diagram
\begin{align}
\label{eq:main-im-diag}
\cdots \emb \im(f^2) \emb \im(f) \emb \im(f^0) = X
\end{align}
has a limit cone $\bigl( L \embby{j_n} \im(f^n) \bigr)_{n \in \Z}$, where
$j_n$ is an embedding by Lemma~\ref{lemma:proj-emb}. Dually, 
\[
X = \im(f^0) \cov \im(f) \cov \im(f^2) \cov \cdots
\]
has a colimit cone $\bigl( \im(f^n) \covby{k_n} M \bigr)_{n \in \Z}$. 

We will show that $L$ is also a limit of $\cdots \toby{f} X \toby{f}
\cdots$. It will follow by duality that $M$ is its colimit, and we then
show that the canonical map $L \to M$ is an isomorphism.

First we construct an automorphism of $L$. Taking limits in
diagram~\eqref{eq:main-ladder}, there is a unique map $\hat{f} \from L \to
L$ such that
\begin{align}
\label{eq:main-auto}
\begin{array}{c}
\xymatrix{
L \ar@{ >->}[r]^-{j_n} \ar[d]_{\hat{f}} &
\im(f^n) \ar@{->>}[d]   \\
L \ar@{ >->}[r]_-{j_{n+1}}      &
\im(f^{n + 1})
}
\end{array}
\end{align}
commutes for all $n \in \N$. By axiom~\axIIIs, $\hat{f}$ is a covering,
which by axiom~\axIs\ implies that $\hat{f}$ is an automorphism of $L$.

Next observe that the family of maps
\begin{align}
\label{eq:main-cone}
\Bigl( L \toby{\hat{f}^n} L \toby{j_0} X \Bigr)_{n \in \Z}
\end{align}
is a cone on $\cdots \toby{f} X \toby{f} \cdots$, since the diagram
\[
\xymatrix@R-2ex@C-2ex{
&
&
L \ar[ld]_{\hat{f}^n} \ar[rd]^{\hat{f}^{n + 1}} &
&
\\
&
L \ar[rr]^{\hat{f}} \ar@{ >->}[ldd]_{j_0}        &
&
L \ar@{ >->}[ld]^{j_1} \ar@{ >->}[rdd]^{j_0}      &
\\
&
&
\im(f) \ar@{ >->}[rrd]   &       
&
\\
X \ar@{->>}[rru] \ar[rrrr]_f    &
&
&
&
X
}
\]
commutes for each $n \in \Z$. 

We will prove that~\eqref{eq:main-cone} is a limit cone. Take
an arbitrary cone $\bigl( A \toby{s_n} X \bigr)_{n \in \Z}$ on
$\cdots \toby{f} X \toby{f} \cdots$. We must show there is a unique
map $\bar{s} \from A \to L$ such that the diagram
\begin{align}
\label{eq:main-lim-comm}
\begin{array}{c}
\xymatrix{
A \ar[rrd]^{s_n} \ar[d]_{\bar{s}}       &
&
\\
L \ar[r]_{\hat{f}^n}    &
L \ar@{ >->}[r]_{j_0}   &
X
}
\end{array}
\end{align}
commutes for all $n \in \Z$. 

For uniqueness, take such a map $\bar{s}$. Then for all $n \in \N$, the
diagram
\[
\xymatrix{
A \ar[rrd]^{s_{-n}} \ar[d]_{\bar{s}}    &
&
\\
L \ar[r]^{\hat{f}^{-n}} \ar[rd]_1       &
L \ar@{ >->}[r]_{j_0} \ar[d]^{\hat{f}^n}        &
X \ar@{->>}[d]  \\
&
L \ar@{ >->}[r]_-{j_n}   &
\im(f^n)
}
\]
commutes, the inner square by diagram~\eqref{eq:main-auto} and
induction. Since the cone $\bigl( L \embby{j_n} \im(f^n)\bigr)_{n \in \N}$
is a limit, this property determines $\bar{s}$ uniquely.

For existence, consider the diagram
\[
\xymatrix{
A \ar@/^/[rrrd]|-*+<3pt>{\scriptstyle s_0} 
\ar@/^/[rrd]|-*+<2pt>{\scriptstyle s_{-1}} 
\ar@/^/[rd]|-*+<2pt>{\scriptstyle s_{-2}}     
&
&
&
\\
\cdots \ar[r]^f &
X \ar[r]^f \ar@{->>}[d] &
X \ar[r]^(0.4)f \ar@{->>}[d] &
X \ar@{=}[d]    \\
\cdots \ar@{ >->}[r]    &
\im(f^2) \ar@{ >->}[r]  &
\im(f) \ar@{ >->}[r]    &
X.
}
\]
The upper part commutes by definition of cone, and the lower part commutes
by the leftmost square of diagrams~\eqref{eq:in-out} in the case $k =
1$. Hence
\[
\bigl( A \toby{s_{-n}} X \cov \im(f^n) \bigr)_{n \in \N}
\]
is a cone on~\eqref{eq:main-im-diag}. There is, therefore, a unique map
$\bar{s} \from A \to L$ such that
\begin{align}
\label{eq:med-sq}
\begin{array}{c}
\xymatrix{
A \ar[r]^{s_{-n}} \ar[d]_{\bar{s}}      &
X \ar@{->>}[d]  \\
L \ar@{ >->}[r]_-{j_n}  &
\im(f^n) 
}
\end{array}
\end{align}
commutes for all $n \in \N$. Our task is to show that
diagram~\eqref{eq:main-lim-comm} commutes for all $n \in \Z$. Now for each
$n \in \Z$, there is a cone
\[
\xymatrix{
A \ar@/^/[rrrd]|-*+<3pt>{\scriptstyle s_n} 
\ar@/^/[rrd]|-*+<2pt>{\scriptstyle s_{n-1}} 
\ar@/^/[rd]|-*+<2pt>{\scriptstyle s_{n-2}}   &
&
&
\\
\cdots \ar[r]^f &
X \ar[r]^f \ar@{->>}[d] &
X \ar[r]^(0.4)f \ar@{->>}[d] &
X \ar@{=}[d]    \\
\cdots \ar@{ >->}[r]    &
\im(f^2) \ar@{ >->}[r]  &
\im(f) \ar@{ >->}[r]    &
X
}
\]
on $\cdots \emb \im(f) \emb X$, so there is a unique map $\bar{s}_n \from A
\to L$ such that for all $m \geq 0$,
\begin{align}
\label{eq:main-shift-sq}
\begin{array}{c}
\xymatrix{
A \ar[r]^{s_{n - m}} \ar[d]_{\bar{s}_n}         &
X \ar@{->>}[d]  \\
L \ar@{ >->}[r]_-{j_m}  &
\im(f^m)
}
\end{array}
\end{align}
commutes. In particular, $\bar{s}_0 = \bar{s}$. 

I claim that $\bar{s}_{n + 1} = \hat{f} \of \bar{s}_n$ for all $n \in
\Z$. By the limit property of $L$, it is enough to prove that for each $m
\geq 1$, the outside of the diagram
\[
\xymatrix@R-2ex@C-2ex{
A \ar[rr]^{\bar{s}_{n + 1}} \ar[dd]_{\bar{s}_n} \ar[rd]^{s_{n - m + 1}} &
&
L \ar@{ >->}[ddd]^{j_m} \\
&
X \ar@{->>}[rdd] \ar@{->>}[d]   &
\\
L \ar@{ >->}[r]_-{j_{m - 1}} \ar[d]_{\hat{f}}       &
\im(f^{m - 1}) \ar@{->>}[rd]    &
\\
L \ar@{ >->}[rr]_{j_m} &
&
\im(f^m)
}
\]
commutes. The inner polygons commute, the squares being cases
of~\eqref{eq:main-shift-sq} and~\eqref{eq:main-auto}, so the claim is
proved. 

It follows that for all $n \in \Z$, the left-hand triangle of
\[
\xymatrix{
A \ar[rrd]^{s_n} 
\ar[rd]|*+<3pt>{\scriptstyle\bar{s}_n} 
\ar[d]_{\bar{s}}   &
&
\\
L \ar[r]_{\hat{f}^n}    &
L \ar@{ >->}[r]_{j_0}   &
X
}
\]
commutes. The right-hand triangle also commutes, being the case $m = 0$
of diagram~\eqref{eq:main-shift-sq}. Hence the outside, which
is diagram~\eqref{eq:main-lim-comm}, commutes. This completes the proof
that $\bigl( L \toby{\hat{f}^n} L \embby{j_0} X \bigr)_{n \in \Z}$ is a
limit cone on $\cdots \toby{f} X \toby{f} \cdots$. 

Dually, $\bigl( X \covby{k_0} M \toby{\hat{f}^n} M \bigr)_{n \in \Z}$ is a
colimit cone on the same diagram.

Next we show that the composite $L \embby{j_0} X \covby{k_0} M$ is
an isomorphism. By~\eqref{eq:main-auto}, the diagram
\[
\xymatrix{
L \ar[r]^{\hat{f}}_\sim \ar@{ >->}[d]_{j_0}     &
L \ar[r]^{\hat{f}}_\sim \ar@{ >->}[d]_{j_1}     &
L \ar[r]^{\hat{f}}_\sim \ar@{ >->}[d]_{j_2}     &
\cdots  \\
X \ar@{->>}[r]  &
\im(f) \ar@{->>}[r]     &
\im(f^2) \ar@{->>}[r]   &
\cdots,
}
\]
commutes. The top row has colimit $L$
with $n$th coprojection $\hat{f}^{-n}$; the bottom row has colimit
$M$. Write $\phi \from L \to M$ for the induced map, which is unique such
that 
\[
\xymatrix{
L \ar[r]^{\hat{f}^{-n}} \ar@{ >->}[d]_{j_n}     &
L \ar[d]^\phi   \\
\im(f^n) \ar@{->>}[r]_-{k_n}    &
M
}
\]
commutes for all $n \in \N$. In particular, it commutes for $n = 0$, so
$\phi = k_0 j_0$. But by axiom~\axIIIs, $\phi$ is an embedding, so $k_0
j_0$ is an embedding. By duality, $k_0 j_0$ is also a covering. Hence $k_0
j_0$ is an isomorphism, as claimed.

We have shown that $\ento{X}{f}$ has eventual image duality and that
$\ei{f}$ can be constructed as either the limit $L$ of $\cdots \emb \im(f)
\emb X$ (with $\iota_f$ as the $0$th projection $j_0$) or the colimit $M$
of $X \cov \im(f) \cov \cdots$ (with $\pi_f = k_0$).
It only remains to prove that the map $\hat{f}$ induced on limits by the map
of diagrams~\eqref{eq:main-ladder} is $\au{f}$; the dual statement on
colimits will follow by duality. For this, we must prove that the outside of
the square
\[
\xymatrix{
L \ar[r]^{\pr_n} 
\ar[rd]|*+<3pt>{\scriptstyle\pr_{n + 1}} 
\ar[d]_{\hat{f}} &
X \ar[d]^f      \\
L \ar[r]_{\pr_n}        &
X
}
\]
commutes for each $n \in \Z$, where $\pr_n = j_0 \of \hat{f}^n$ is the
$n$th projection of the limit cone just constructed. The lower triangle
commutes by definition of $\pr_n$, and the upper triangle since
$(\pr_m)_{m \in \Z}$ is a cone. This completes the proof.
\end{proof}

\section{The eventual image is a terminal coalgebra}
\label{sec:coalg}

In our three main example categories, the eventual image of an endomorphism
$\ento{X}{f}$ is the largest subspace $A$ of $X$
satisfying $A \sub fA$. Here, we generalize this statement to categories
with a factorization system of finite type.

The general result will be expressed in terms of terminal coalgebras.
Recall that given an endofunctor $T$ of a category $\cat{A}$, a
\demph{$T$-coalgebra} is a pair $(A, \alpha)$ with $A \in \cat{A}$ and
$\alpha\from A \to TA$. With the obvious maps, $T$-coalgebras form a
category. The \emph{terminal} $T$-coalgebra, if it exists, plays an
important role, and Lambek showed that it is a fixed point:
its structure map $\alpha$ is an isomorphism (Lemma~2.2 of~\cite{LambFTC}).

Coalgebras in this sense arise in many situations in mathematics and
computer science, typically involving infinite iteration or coinduction. To
give just two examples, bisimulation in the context of Milner's concurrency
theory can be described in terms of coalgebras~\cite{AcMe}, and weak
$\infty$-categories can be defined using terminal
coalgebras~\cite{WICVTC}. See Ad\'amek~\cite{AdamIC} and Rutten~\cite{Rutt}
for surveys.

Let $\C$ be a category with a factorization system. Let $X \in \C$. The
slice category $\C/X$ has a full subcategory $\Emb(X)$ consisting of the
embeddings into $X$. A map from $A \embby{j} X$ to $B \embby{k}
X$ in $\Emb(X)$ is, then, a map $u \from A \to B$ in $\C$ such that
\begin{align}
\label{eq:slice-map}
\begin{array}{c}
\xymatrix@R-2.5ex@C-2ex{
A \ar[rr]^u \ar@{ >->}[rd]_j    &
&
B \ar@{ >->}[ld]^k      \\
&
X
}
\end{array}
\end{align}
commutes, and the two out of three property implies that $u$ is also an
embedding. 

Given also an endomorphism $f$ of $X$, there is an endofunctor $f_!$ of
$\Emb(X)$ defined as follows. For an object $A \embby{j} X$, take the
image factorization
\[
\xymatrix{
A \ar@{ >->}[r]^j \ar@{->>}[d]  &
X \ar[d]^f      \\
fA \ar@{ >->}[r]_{j^\#}   &
X
}
\]
of $fj$ (where $fA$ is alternative notation for $\im(fj)$) and define
\[
f_!\Bigl(A \embby{j} X \Bigr) = \Bigl( f A \embby{j^\#} X\Bigr).
\]
For a
map~\eqref{eq:slice-map} in $\Emb(X)$, the solid part of the diagram
\[
\xymatrix@R-2ex@C-2ex{
A \ar@{ >->}[rr]^u \ar@{ >->}[rd]_j \ar@{->>}[dd]       &
&
B \ar@{ >->}[ld]^k \ar@{->>}[dd]        \\
&
X \ar[dd]^(0.3)f        &
\\
fA \ar@{.>}'[r]^(0.6){f_!(u)}[rr] \ar@{ >->}[rd]_{j^\#}   &
&
fB \ar@{ >->}[ld]^{k^\#}  \\
&
X       &
}
\]
commutes, so by orthogonality
(Lemma~\ref{lemma:fs-basic}\bref{part:fsb-orth}), there is a unique map
$f_!(u) \from fA \to fB$ making the diagram commute. Then $f_!(u)$ is a map
from $f_!\bigl(A \embby{j} X\bigr)$ to $f_!\bigl(B \embby{k} X\bigr)$ in
$\Emb(X)$. This defines an endofunctor $f_!$ of $\Emb(X)$.

\begin{example}
\label{eg:ei-coalg}
Let $\ento{X}{f}$ be an endomorphism in $\C$ with eventual image duality. The
diagram 
\[
\xymatrix{
\ei{f} \ar@{ >->}[r]^-{\iota_f} \ar[d]_{\au{f}}^{\iso}  &
X \ar[d]^f      \\
\ei{f} \ar@{ >->}[r]_-{\iota_f} &
X 
}
\]
commutes by definition of $\au{f}$, so $f_!$ fixes the object $\ei{f}
\embby{\iota_f} X$ of $\Emb(X)$. Together with the identity, this object
is a coalgebra for $f_!$, which we call just $\ei{f}$.
\end{example}

We prove that $\ei{f}$ is the terminal $f_!$-coalgebra using a standard
result generally attributed to Ad\'amek~\cite{AdamFAA}; see also
\cite{AdamIC}, Corollary~3.18. 

\begin{thm}[Ad\'amek]
Let $T$ be an endofunctor of a category $\cat{A}$. Suppose that $\cat{A}$
has a terminal object $1$, that the diagram
\begin{align}
\label{eq:ad-seq}
\cdots \toby{T^2!} T^2 1 \toby{T!} T1 \toby{!} 1
\end{align}
has a limit $\bigl( L \toby{j_n} T^n 1\bigr)_{n \in \N}$ in $\cat{A}$
(where $!$ is the unique map $T1 \to 1$), and that this limit is preserved
by $T$. Write $\lambda$ for the canonical isomorphism $TL \to L$. Then $(L,
\lambda^{-1})$ is the terminal $T$-coalgebra.
\end{thm}

Here $\lambda$ is the unique map $TL \to L$ such that
\[
\xymatrix{
TL \ar[r]^\lambda \ar@/_/[rrd]_{T(j_n)}     &
L \ar[rd]^{j_{n + 1}}   &       \\
        &       &T^{n + 1}1
}
\]
commutes for all $n \in \N$, which is an isomorphism since $T$ preserves
the limit. 

To apply Ad\'amek's theorem, we use the following observation.

\begin{lemma}
\label{lemma:creation}
Let $\C$ be a category with a factorization system of finite type, and let
$X \in \C$. Then $\Emb(X)$ has, and the forgetful functor $\Emb(X) \to \C$
creates, sequential limits.
\end{lemma}

\begin{proof}
Take a diagram
\begin{align}
\label{eq:cr-seq}
\cdots \toby{u_1} 
\lefts
\begin{array}{c}
\xymatrix{A_1 \ar@{ >->}[d]^{i_1} \\ X} 
\end{array}
\rights  \toby{u_0}
\lefts 
\begin{array}{c}
\xymatrix{A_1 \ar@{ >->}[d]^{i_0} \\ X} 
\end{array}
\rights
\end{align}
in $\Emb(X)$. We show that $\cdots \toby{u_1} A_1 \toby{u_0} A_0$ has
a limit cone in $\C$ and that any such cone lifts uniquely to a cone in
$\Emb(X)$, which is also a limit cone.

By the two out of three property, each $u_n$ is an embedding. Hence the
diagram $\cdots \embby{u_1} A_1 \embby{u_0} A_0$ has a limit cone $\bigl( L
\toby{j_n} A_n \bigr)_{n \in \N}$ in $\C$, and by
Lemma~\ref{lemma:proj-emb}, each $j_n$ is an embedding. The forgetful
functor $\C/X \to \C$ strictly creates connected limits, so there is a
unique map $k \from L \to X$ such that
\begin{align}
\label{eq:cr-cone}
\left(
\lefts
\begin{array}{c}
\xymatrix{L \ar[d]^k\\ X}
\end{array}
\rights
\xymatrix{ \ \ar@{ >->}[r]^{j_n} & \ }
\lefts
\begin{array}{c}
\xymatrix{A_n \ar@{ >->}[d]^{i_n} \\ X}
\end{array}
\rights
\right)_{n \in \N}
\end{align}
is a limit cone in $\C/X$. Then $k = i_0 j_0$, and $i_0$ and $j_0$ are
embeddings, so $k$ is too. Hence~\eqref{eq:cr-cone} is a limit cone
on~\eqref{eq:cr-seq} in $\Emb(X)$. 
\end{proof}

\begin{thm}
\label{thm:coalg}
Let $\C$ be a category with a factorization system of finite type. Let
$\ento{X}{f}$ be an endomorphism in $\C$. Then $\ei{f}$ is
the terminal coalgebra for the endofunctor $\ento{\Emb(X)}{f_!}$.
\end{thm}

\begin{proof}
We use Ad\'amek's theorem, first showing that the diagram~\eqref{eq:ad-seq}
is in this case
\begin{align}
\label{eq:coalg-seq}
\cdots 
\xymatrix{ \ \ar@{ >->}[r] & \ }
\lefts
\begin{array}{c}
\xymatrix{\im(f^2) \ar@{ >->}[d]\\ X}
\end{array}
\rights
\xymatrix{ \ \ar@{ >->}[r] & \ }
\lefts
\begin{array}{c}
\xymatrix{\im(f) \ar@{ >->}[d]\\ X}
\end{array}
\rights
\xymatrix{ \ \ar@{ >->}[r] & \ }
\lefts
\begin{array}{c}
\xymatrix{X \ar[d]^1 \\ X}
\end{array}
\rights.
\end{align}
The terminal object of $\Emb(X)$ is $(X \toby{1} X)$. That $f_!^n$ applied
to the terminal object is $\im(f^n) \emb X$ follows by induction from the
rightmost square of diagrams~\eqref{eq:in-out} with $k = 1$. Now
assume inductively that the map $T^n!$ of diagram~\eqref{eq:ad-seq} is
the embedding
 \begin{align}
\label{eq:slice-term-n}
\lefts
\begin{array}{c}
\xymatrix{ \im(f^{n + 1}) \ar@{ >->}[d] \\ X } 
\end{array}
\rights
\ 
\xymatrix{ \ \ar@{ >->}[r] & \ }
\ 
\lefts
\begin{array}{c}
\xymatrix{ \im(f^n) \ar@{ >->}[d] \\ X } 
\end{array}
\rights.
\end{align}
By definition, $f_!$ applied to the map~\eqref{eq:slice-term-n} is the
unique dotted map making the diagram
\[
\xymatrix@C-2ex{
\im(f^{n + 1}) \ar@{ >->}[rr] \ar@{->>}[d]      &
&
\im(f^n) \ar@{->>}[d]   \\
\im(f^{n + 2}) \ar@{.>}[rr] \ar@{ >->}[rd]      &
&
\im(f^{n + 1}) \ar@{ >->}[ld]   \\
&
X
}
\]
commute. But by Lemma~\ref{lemma:emb-cov}, the embedding $\im(f^{n + 2})
\emb \im(f^{n + 1})$ makes this diagram commute, completing the
induction.

Theorem~\ref{thm:main}\bref{part:main-lim} gives a limit cone 
\[
\bigl( \ei{f} \embby{j_n} \im(f^n) \bigr)_{n \in \N}
\]
on $\cdots \emb \im(f) \emb X$, and then
\begin{align}
\label{eq:coalg-cone}
\left(
\lefts
\begin{array}{c}
\xymatrix{ \ei{f} \ar@{ >->}[d]^{\iota_f = j_0} \\ X}
\end{array}
\rights
\xymatrix{\ \ar@{ >->}[r]^{j_n} &\ }
\lefts
\begin{array}{c}
\xymatrix{ \im(f^n) \ar@{ >->}[d] \\ X }
\end{array}
\rights
\right)_{n \in \N}
\end{align}
is a cone on~\eqref{eq:coalg-seq}. By Lemma~\ref{lemma:creation}, it is a
limit cone.

It remains to show that this limit is preserved by $f_!$, and for this, it
is enough to prove that $f_!$ maps the cone~\eqref{eq:coalg-cone} to
\[
\left(
\lefts
\begin{array}{c}
\xymatrix{ \ei{f} \ar@{ >->}[d]^{\iota_f} \\ X}
\end{array}
\rights
\xymatrix{\ \ar@{ >->}[r]^{j_{n + 1}} &\ }
\lefts
\begin{array}{c}
\xymatrix{ \im(f^{n + 1}) \ar@{ >->}[d] \\ X }
\end{array}
\rights
\right)_{n \in \N}.
\]
We have already shown that $f_!$ fixes the object $\ei{f} \embby{\iota_f}
X$ of $\Emb(X)$ and that it maps $\im(f^n) \emb X$ to $\im(f^{n + 1})
\emb X$. Moreover, for each $n \in \N$ we have a commutative diagram
\[
\xymatrix{
\ei{f} \ar@{ >->}[rr]^{j_n} \ar[d]_{\au{f}}^{\iso}      &
&
\im(f^n) \ar@{->>}[d]   \\
\ei{f} \ar@{ >->}[rr]^{j_{n + 1}} \ar@{ >->}[rd]_{\iota_f}      &
&
\im(f^{n + 1}) \ar@{ >->}[ld]   \\
&
X,
}
\]
where the square commutes by Theorem~\ref{thm:main}\bref{part:main-lim}. By
definition of $f_!$ on morphisms, this means that $f_!(j_n) = j_{n + 1}$,
as required.
\end{proof}

The dual result characterizes the eventual image as the initial algebra for
an endofunctor on the category of covering maps out of $X$.

\section{Finite sets}
\label{sec:set}

The category $\FinSet$ of finite sets has a factorization system consisting
of injections and surjections. It is of finite type:
axioms~\axI\ and~\axIs\ state that any injective or surjective endomorphism
of a finite set is invertible, and the rest of the axioms are trivial
because any diagram
\[
\cdots \emb X_1 \emb X_0
\qquad \text{or} \qquad
X_0 \cov X_1 \cov \cdots
\]
in $\FinSet$ stabilizes after a finite number of steps. So
Theorems~\ref{thm:main} and~\ref{thm:coalg} apply, showing that $\FinSet$
has eventual image duality and providing characterizations of the
eventual image, which we now study in detail.

For the rest of this section, let $f$ be an endomorphism of a finite set
$X$ (Figure~\ref{fig:set-endo}).

\begin{figure}
\lengths
\begin{picture}(120,40)
\cell{60}{20}{c}{\includegraphics[width=120\unitlength]{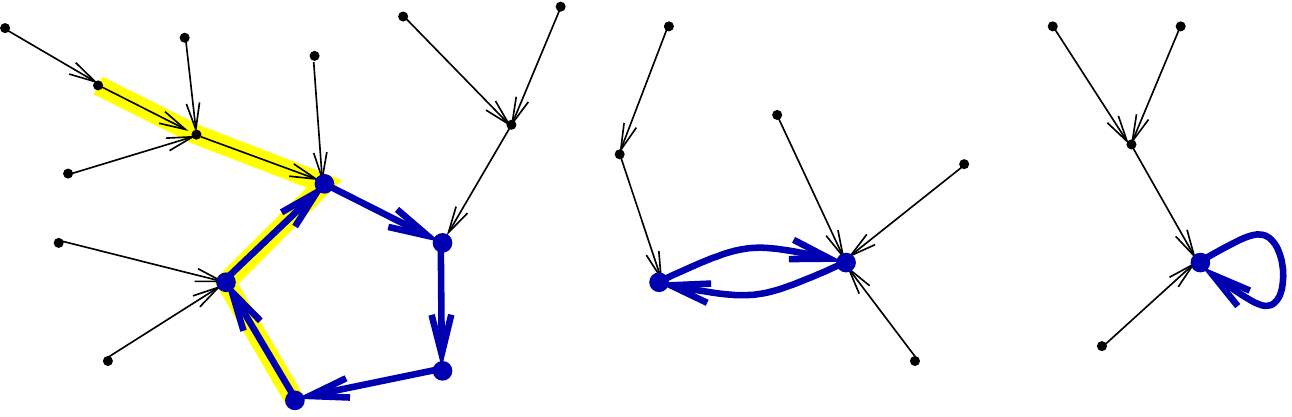}}
\cell{7.5}{29.5}{c}{$x$}
\cell{21}{-0.5}{b}{$f^\infty(x)$}
\end{picture}%
\caption{An endomorphism $f$ of a finite set. The eventual image and its
automorphism $\au{f}$ are shown in bold blue. A point $x$ and the resulting
point $f^\infty(x)$ are shown, with the yellow highlighted path
illustrating the back and forth description of $f^\infty$.}
\label{fig:set-endo}
\end{figure}

The eventual image of $f$ is both the limit and colimit of the diagram 
\begin{align}
\label{eq:set-seq}
\cdots \toby{f} X \toby{f} X \toby{f} X \toby{f} \cdots.
\end{align}
Example~\ref{eg:evim-set} gives explicit descriptions of the limit, the
colimit, and the canonical map from limit to colimit. That the
canonical map is bijective means that for any $N \in \Z$ 
and $x \in X$, there is a unique double sequence $(y_n)_{n \in \Z}$ such that $f(y_n) = y_{n +
1}$ for all $n$ and $y_n = f^{n - N}(x)$ for all sufficiently large
$n$. Writing $x_n = f^{n - N}(x)$, this condition can be depicted as follows:
\[
\xymatrix@C-.2em{
\cdots \ar@{|->}[r]     &
y_N \ar@{|->}[r]        &
\cdots \ar@{|->}[r]  &
y_{p - 1} \ar@{|->}[rd]    &
&
\\
&
x = x_N \ar@{|->}[r]        &
\cdots \ar@{|->}[r]  &
x_{p - 1} \ar@{|->}[r]     &
x_p = y_p \ar@{|->}[r]  &
x_{p + 1} = y_{p + 1} \ar@{|->}[r]      &
\cdots
}
\]

Theorem~\ref{thm:main}\bref{part:main-lim} implies that $\ei{f} \iso
\bigcap_{n \in \N} \im(f^n)$.  The chain of inclusions
\begin{align}
\label{eq:nested-sets}
\cdots \sub \im(f^2) \sub \im(f) \sub X
\end{align}
contains at most $\mg{X}$ proper inclusions, and if any inclusion is an
equality then so are all the inclusions to its left. Hence the sequence
stabilizes after at most $\mg{X}$ steps and $\ei{f} = \im(f^{\mg{X}})$.
The canonical isomorphism from the limit of~\eqref{eq:set-seq} to
$\bigcap_{n \in \N} \im(f^n)$ is $(x_n)_{n \in \Z} \mapsto x_0$.

Dually, Theorem~\ref{thm:main}\bref{part:main-colim} implies that $\ei{f}
\iso X/{\sim}$, where $x \sim y$ if $f^n(x) = f^n(y)$ for some $n \in
\N$. By a similar stabilization argument, $x \sim y$ if and only if
$f^{\mg{X}}(x) = f^{\mg{X}}(y)$. The canonical isomorphism from $X/{\sim}$
to the colimit of~\eqref{eq:set-seq} maps the equivalence class of $x$ to
the equivalence class of $(0, x)$.

Let us temporarily write $\ei{f}$ and $\edi{f}$ for the limit and colimit
of~\eqref{eq:set-seq}, respectively. The map $\iota_f \from \ei{f} \to X$
is $(x_n)_{n \in \Z} \mapsto x_0$, and the map $\pi_f \from X \to \edi{f}$
is $x \mapsto [(0, x)]$. Their composite is the canonical map
\[
\begin{array}{cccc}
\can \from      &\ei{f}                 &\to            &\edi{f}        \\
                &(x_n)_{n \in \Z}       &\mapsto        &[(0, x_0)],
\end{array}
\]
which is a bijection. Now consider the idempotent $f^\infty$, which is the
composite
\begin{equation}
\label{eq:set-idem-comp}
X \covby{\pi_f} \edi{f} \toby{\can^{-1}} \ei{f} \embby{\iota_f} X.
\end{equation}

\begin{propn}
\label{propn:set-trace}
Let $x \in X$. Then for all $n \geq \mg{X}$, we have $f^n(x) \in \ei{f}$
and $f^\infty(x) = \au{f}^{-n}(f^n(x))$.
\end{propn}

Typically we treat $\ei{f}$ as a subset of $X$, as in this
statement.

\begin{proof}
Let $x \in X$. The effect of the maps~\eqref{eq:set-idem-comp} on $x$ is
\[
x \mapsto
[(0, x)] \mapsto
(y_n)_{n \in \Z} \mapsto
y_0 = f^\infty(x),
\]
where $(y_n)_{n \in \Z}$ is the unique double sequence in $X$ such that
$f(y_n) = y_{n + 1}$ for all $n \in \Z$ and $y_n = f^n(x)$ for all
sufficiently large $n$. Equivalently, `for all sufficiently large $n$' can
be replaced by `for all $n \geq \mg{X}$' (by the description of $\sim$
above). Hence for all $n \geq \mg{X}$, 
\begin{equation}
\label{eq:st-final}
f^n(f^\infty(x)) = f^n(y_0) = y_n = f^n(x)
\end{equation}
Since $f^\infty(x)$ is in the
subset $\ei{f}$ of $X$, which is $f$-invariant, $f^n(x) \in
\ei{f}$. Finally, applying $\au{f}^{-n}$ to~\eqref{eq:st-final} gives
$f^\infty(x) = \au{f}^{-n}(f^n(x))$.
\end{proof}

This result gives a back and forth algorithm for computing $f^\infty(x)$
(Figure~\ref{fig:set-endo}): apply $f$ to $x$ enough times to put it into
the eventual image, then apply $\au{f}^{-1}$ the same number of times.

The back and forth description of $f^\infty$ is well known in finite
semigroup theory, and further light is shed by a standard result
(Corollary~1.2 of~\cite{SteiRTF}):

\begin{lemma}
\label{lemma:semi-idem}
Let $S$ be a finite semigroup and $\sigma \in S$. Then the set $\{\sigma,
\sigma^2, \ldots\}$ contains exactly one idempotent.
\end{lemma}

\begin{proof}
Since $S$ is finite, there exist $m, k \geq 1$ such that $\sigma^m =
\sigma^{m + k}$. Then $\sigma^{n + rk} = \sigma^n$ for all $n \geq m$ and
$r \geq 0$. Hence $\sigma^{mk}$ is idempotent. Moreover,
if both $\sigma^p$ and $\sigma^q$ are idempotent ($p, q \geq 1$) then
$\sigma^p = (\sigma^p)^q = (\sigma^q)^p = \sigma^q$. 
\end{proof}

Next we show that $f^\infty$ is a finite power of $f$.

\begin{propn}
\label{propn:sets-unique-idem}
The endomorphism $f^\infty$ is the unique idempotent element of $\{f, f^2,
\ldots\}$. In fact, $f^\infty = f^{\mg{X}!}$. 
\end{propn}

\begin{proof}
Since $\au{f}$ is a permutation of a finite set, $\au{f}^{-1} = \au{f}^{\,r}$
for some $r \geq 0$. Then by Proposition~\ref{propn:set-trace}, $f^\infty =
f^{(r + 1)\mg{X}} \in \{f, f^2, \ldots\}$. 

To prove that $f^\infty = f^{\mg{X}!}$, we refine the proof of
Lemma~\ref{lemma:semi-idem}. Let $x \in X$. Since $x, f(x), \ldots,
f^{\mg{X}}(x)$ are not all distinct, $f^m(x) = f^{m + k}$ for some $m \geq
0$ and $k \geq 1$ such that $m + k \leq \mg{X}$. Then $f^{n + rk}(x) =
f^n(x)$ for all $n \geq m$ and $r \geq 0$. In particular, $f^{2n}(x) =
f^n(x)$ whenever $n \geq m$ and $k \dvd n$. It follows that $f^{\mg{X}!}$
is idempotent, so by the uniqueness part of Lemma~\ref{lemma:semi-idem},
$f^\infty = f^{\mg{X}!}$. 
\end{proof}

Theorem~\ref{thm:coalg} describes $\ei{f}$ as the terminal coalgebra for
the endofunctor $f_!$ of $\Emb(X)$. In this case, $\Emb(X)$ is equivalent
to the power set of $X$, and Theorem~\ref{thm:coalg} states that $\ei{f}$
is the largest subset $A$ of $X$ satisfying $A \sub fA$. The dual theorem
states that $\ei{f} = X/{\sim}$, where $\sim$ is the finest equivalence
relation on $X$ such that $f(x) \sim f(y) \implies x \sim y$ for all $x, y
\in X$. 

A point $x \in X$ is \demph{periodic} for $f$ if $x \in \{f(x), f^2(x),
\ldots\}$. As Figure~\ref{fig:set-endo} suggests:

\begin{propn}
The set of periodic points of $f$ is $\ei{f}$.
\end{propn}

\begin{proof}
Let $A$ be the set of periodic points. By definition, $A \sub fA$, so
Theorem~\ref{thm:coalg} implies that $A \sub \ei{f}$. Conversely, every
element $x \in \ei{f}$ is periodic: $\au{f}$ is a permutation of the finite
set $\ei{f}$, so $\au{f}^{\,n} = 1$ for some $n \geq 1$, and then $x =
\au{f}^{\,n}(x) = f^n(x)$.
\end{proof}

\section{Finite-dimensional vector spaces}
\label{sec:vect}

Let $\FDVect$ be the category of finite-dimensional vector spaces over
a field $k$. The injective and surjective linear maps form a factorization
system of finite type: axioms~\axI\ and~\axIs\ hold because any
injective or surjective endomorphism is invertible (by the rank-nullity
formula), and the other axioms hold because any nested sequence of
subspaces or quotient spaces must stabilize after a finite number of
steps. By Theorem~\ref{thm:main}, $\FDVect$ has eventual image duality.

For the rest of this section, let $f$ be an endomorphism of a
finite-dimensional vector space $X$. 

As well as the eventual image of $f$, we consider its \demph{eventual
kernel} $\ek{f}$, the union of the nested sequence
\begin{align}
\label{eq:ker-nested}
\{0\} = \ker(f^0) \sub \ker(f^1) \sub \ker(f^2) \sub \cdots.
\end{align}
A standard lemma states: 

\begin{lemma}[Fitting]
\label{lemma:fitting}
$X = \ei{f} \oplus \ek{f}$. Moreover, $f$ restricts to an
automorphism of $\ei{f}$ and a nilpotent operator on $\ek{f}$. This is the
unique decomposition of $f$ as the direct sum of an automorphism and a
nilpotent. 
\end{lemma}

\begin{proof}
The first part is Theorem~8.5 of~\cite{AxleLAD}, and the other
parts follow.
\end{proof}

Theorem~\ref{thm:main}\bref{part:main-lim} shows that $\ei{f} = \bigcap_{n
\in \N} \im(f^n)$ and that $\iota_f \from \ei{f} \emb X$ is the
inclusion. On the other hand, Theorem~\ref{thm:main}\bref{part:main-colim}
shows that $\ei{f} = X/{\sim}$, where 
\[
x \sim y \iff f^n(x) = f^n(y) \text{ for some } n \in \N.
\]
Equivalently, $\ei{f} = X/\ek{f}$. Moreover, $\pi_f \from X \cov \ei{f}$ is
the quotient map $X \cov X/\ek{f}$.

The canonical map from the limit to the colimit is 
\begin{align}
\label{eq:vect-iso}
\begin{array}{cccc}
\pi_f \iota_f \from     &\ei{f} &\to            &X/\ek{f}       \\
                        &y      &\mapsto        &y + \ek{f}.
\end{array}
\end{align}
That this is an isomorphism means that for all $x \in X$, there is a unique
$y \in \ei{f}$ such that $x + \ek{f} = y + \ek{f}$. This is equivalent to
the statement that $X = \ei{f} \oplus \ek{f}$
(Lemma~\ref{lemma:fitting}). By definition, $y = f^\infty(x)$. Thus,
$f^\infty$ is the projection to $\ei{f}$ associated with the decomposition
$X = \ek{f} \oplus \ei{f}$. In particular, $\ek{f} = \ker(f^\infty)$.

Much as for finite sets, the chain of inclusions~\eqref{eq:nested-sets}
must stabilize after at most $\dim X$ steps, so that $\ei{f} = \im(f^{\dim
X})$. Similarly, the chain of inclusions~\eqref{eq:ker-nested} stabilizes
after at most $\dim X$ steps, so $\ek{f} = \ker(f^{\dim X})$.

\begin{remark} 
The canonical isomorphism~\eqref{eq:vect-iso} is \emph{not} the isomorphism
$\im(f^{\dim X}) \to X/\ker(f^{\dim X})$ provided by the first isomorphism
theorem.  Any nontrivial automorphism of a one-dimensional space is a
counterexample.
\end{remark}

There is a back and forth description of $f^\infty$ analogous to
Proposition~\ref{propn:set-trace}:

\begin{propn}
\label{propn:vect-trace}
Let $x \in X$. Then for all $n \geq \dim X$, we have $f^n(x) \in \ei{f}$
and $f^\infty(x) = \au{f}^{-n}(f^n(x))$.
\end{propn}

\begin{proof}
Let $n \geq \dim X$. Then $\im(f^n) = \ei{f}$ and $\ker(f^n) =
\ek{f}$. Write $x = y + z$ with $y = f^\infty(x) \in \ei{f}$ and $z \in
\ek{f}$. We have $f^n(x) \in \ei{f}$ and
\[
\au{f}^{-n}(f^n(x)) 
= 
\au{f}^{-n}(f^n(y) + f^n(z))
=
\au{f}^{-n}(f^n(y))
=
y
=
f^\infty(x).
\]
\end{proof}

The proof of Proposition~\ref{propn:sets-unique-idem} used the fact that the
inverse of an automorphism $g$ of a finite set is a nonnegative power of
$g$. We will need the linear analogue. Write $\chi_g(t) = \det(g - tI)$ for
the characteristic polynomial of an operator $g$.

\begin{lemma}
\label{lemma:ch}
Let $g$ be an automorphism of a finite-dimensional vector space. Then
$g^{-1}$ is a polynomial in $g$; indeed, $g^{-1} = q(g)$ where 
\[
q(t) = \frac{\det g - \chi_g(t)}{(\det g)t} \in k[t].
\]
\end{lemma}

\begin{proof}
Write $\chi_g(t) = \sum_{n \geq 0} a_n t^n$. Since $g$ is an automorphism,
$0 \neq \det g = a_0$. By the Cayley--Hamilton theorem,
\[
0 = a_0 + g \sum_{n \geq 1} a_n g^{n - 1}.
\]
Rearranging shows that
$
\frac{-1}{a_0} \sum_{n \geq 1} a_n g^{n - 1}
$
is inverse to $g$, and the result follows.
\end{proof}

\begin{propn}
$\au{f}^{-1}$ is a polynomial in $\au{f}$.  
\qed
\end{propn}

\begin{propn}
\label{propn:vect-idem-span}
$f^\infty$ is a polynomial in $f$. Indeed, 
\[
f^\infty =
\biggl( 1 - \frac{\chi_{\au{f}}(f)}{\det \au{f}} \biggr)^n
\]
whenever $n \geq \dim X$, and $f^\infty \in \spn\{f, f^2, \ldots\}$ in
$\Hom(X, X)$. 
\end{propn}

\begin{proof}
By Lemma~\ref{lemma:ch}, $\au{f}^{-1} = q(\au{f})$ where
\[
q(t) = \frac{\det \au{f} - \chi_{\au{f}}(t)}{(\det \au{f})t} \in k[t].
\]
Write $r(t) = tq(t)$. Then by Proposition~\ref{propn:vect-trace}, for all
$x \in X$ and $n \geq \dim X$,
\[
f^\infty(x) = q(\au{f})^n(f^n(x)) = q(f)^n(f^n(x)) = r(f)^n(x).
\]
Finally, $f^\infty \in \spn\{f, f^2, \ldots\}$ since $r(t)$ has constant
term $0$.
\end{proof}

\begin{remark}
The proofs of Propositions~\ref{propn:vect-trace}
and~\ref{propn:vect-idem-span} can be refined to weaken the lower bound on
$n$ to $\dim\ek{f}$. 
\end{remark}

In the case of $\FDVect$, Theorem~\ref{thm:coalg} on terminal coalgebras
states that $\ei{f}$ is the largest linear subspace $W$ of $X$ satisfying
$W \sub fW$. The dual of Theorem~\ref{thm:coalg} states that $\ei{f} =
X/\ek{f}$, with $\ek{f}$ characterized as the smallest linear subspace
$U$ of $X$ satisfying $f^{-1}U \sub U$.

Call an element $x \in X$ \demph{linearly periodic} for $f$ if $x \in
\spn\{f(x), f^2(x), \ldots\}$. 

\begin{propn}
The set of linearly periodic points for $f$ is $\ei{f}$. In particular, the
set of linearly periodic points is a linear subspace of $X$.
\end{propn}

\begin{proof}
Let $x$ be a linearly periodic point. Then $x = (p(f) \of f)(x)$ for some
$p(t) \in k[t]$. Hence for all $n \in \N$,
\[
x = (p(f) \of f)^n(x) = f^n(p(f)^n(x)) \in \im(f^n),
\]
giving $x \in \ei{f}$. The converse follows from the last part of
Proposition~\ref{propn:vect-idem-span}. 
\end{proof}

Recall the notion of shift equivalence from Section~\ref{sec:props}. In the
paper by Williams in which it was first introduced (\cite{WillCOD},
p.~342), the following result was proved using zeta functions. Here we give
a different proof.

\begin{propn}
Let $\ento{X}{f}$ and $\ento{Y}{g}$ be shift equivalent endomorphisms in
$\FDVect$. Then the characteristic polynomials of $f$ and $g$ are equal up
to a factor of $\pm t^p$, for some $p \in \Z$. 
\end{propn}

\begin{proof}
By the decomposition in Lemma~\ref{lemma:fitting}, $\chi_f = \chi_{\au{f}}
\cdot \chi_{f_0}$, where $f_0$ is the operator $f$ restricted to
$\ek{f}$. Since $f_0$ is nilpotent, $\chi_{f_0}(t) = \pm t^i$ for some $i
\geq 0$. Hence $\chi_f(t) = \pm t^i \chi_{\au{f}}(t)$. Similarly,
$\chi_g(t) = \pm t^j \chi_{\au{g}}(t)$ for some $j \geq 0$.

By shift equivalence and Proposition~\ref{propn:shift-eqv},
$\ento{\ei{f}}{\au{f}} \iso \ento{\ei{g}}{\au{g}}$. In particular,
$\chi_{\au{f}} = \chi_{\au{g}}$, giving $\chi_f(t) = \pm t^{i - j}
\chi_g(t)$.
\end{proof}

When $k$ is algebraically closed, $X$ decomposes canonically into its
\demph{generalized eigenspaces} $\ek{f - \lambda}$: 
\[
X = \bigoplus_{\lambda \in k} \ek{f - \lambda}
\]
(\cite{AxleLAD}, Theorem~8.21). The dimension
of $\ek{f - \lambda}$ is the algebraic multiplicity of $\lambda$, taken to
be $0$ unless $\lambda$ is an eigenvalue. This decomposition refines the
earlier decomposition $X = \ei{f} \oplus \ek{f}$: 
\[
\ei{f} = \bigoplus_{0 \neq \lambda \in k} \ek{f - \lambda},
\]
providing yet another description of the eventual image.

\section{Compact metric spaces}
\label{sec:met}

Here we study the category $\CptMet$ of compact metric spaces $X = (X, d)$.
Its maps $X \to Y$ are the functions $f\from X \to Y$ that are
\demph{distance-decreasing}: $d(f(x), f(x')) \leq d(x, x')$ for all $x, x'
\in X$. Such a map is an \demph{isometry} if it is distance-preserving:
$d(f(x), f(x')) = d(x, x')$. The isomorphisms in $\CptMet$ are the
surjective isometries.

The isometries and surjections define a factorization system on
$\CptMet$. We will prove that it is of finite type.

\begin{lemma}
A self-isometry of a compact metric space is surjective.
\end{lemma}

This result is classical (Theorem~1.6.14 of~\cite{BBI}), but we give
the short proof.

\begin{proof}
For $\epsln > 0$ and compact $X$, let $N_\epsln(X)$ be the maximal
cardinality of a subset of $X$ that is \demph{$\epsln$-separated}: distinct
points are at least $\epsln$ apart. We show that whenever $Y$ is a compact
proper subspace of $X$, there is some $\epsln > 0$ such that $N_\epsln(Y) <
N_\epsln(X)$. The result follows: for if $\ento{X}{f}$ is an isometry then
$X \iso fX$, so $N_\epsln(X) = N_\epsln(fX)$ for all $\epsln$, so $fX = X$.

Choose $x \in X$ and $\epsln > 0$ such that the ball $B(x,
\epsln)$ is disjoint from $Y$. Choose an $\epsln$-separated set $S$ in $Y$
of cardinality $N_\epsln(Y)$. Then $S \cup \{x\}$ is an $\epsln$-separated
set in $X$, proving that $N_\epsln(Y) < N_\epsln(X)$. 
\end{proof}

This proves axiom~\axI. Axiom~\axIs\ is also standard (Theorem~1.6.15(1)
of~\cite{BBI}), but we give a categorical proof that may have further
applications.

\begin{lemma}
\label{lemma:closed}
Let $\C$ be a category with a factorization system and a closed structure
$([-, -], I)$. Suppose that for all coverings $f \from X \cov Y$ and
objects $Z$, the map $[f, Z] \from [Y, Z] \to [X, Z]$ is an
embedding. Suppose also that every split monic covering is an
isomorphism. Then axiom~\axI\ implies axiom~\axIs.
\end{lemma}

As in Eilenberg and Kelly~\cite{EiKe}, a \demph{closed structure} on $\C$
consists of a functor $[-, -]\from \C^\op \times \C \to \C$ and an object
$I \in \C$ satisfying axioms. For example, $\C$ carries a closed structure
if it is monoidal closed. The axioms on a closed structure imply that
\begin{align}
\label{eq:closed-nat}
\C(X, Y) \iso \C(I, [X, Y])
\end{align}
naturally in $X, Y \in \C$. We write this isomorphism as $f \mapsto
\mapname{f}$.

\begin{proof}
Suppose that axiom~\axI\ holds, and let $f \from X \to X$ be a covering.
Then the endomorphism $[f, X]$ of $[X, X]$ is an embedding, hence, by
axiom~\axI, an isomorphism. There is a unique map $g \from X \to X$
such that
\[
\mapname{g} 
= 
\Bigl( I \toby{\mapname{1_X}} [X, X] \toby{[f, X]^{-1}} [X, X] \Bigr).
\]
The naturality of the isomorphism~\eqref{eq:closed-nat} and the definition
of $g$ give
\[
\mapname{g \of f} = [f, X] \of \mapname{g} = \mapname{1_X},
\]
so $g \of f = 1_X$. But then $f$ is a split monic covering, and,
therefore, an isomorphism.
\end{proof}

As is well known, the category of metric spaces (allowing $\infty$ as a
distance) and distance-decreasing maps has the following symmetric monoidal
closed structure. The tensor product $X \otimes Y$ is the cartesian product
with distances defined by adding the distances in $X$ and $Y$. The unit
object $I$ is the one-point space. The function space $[X, Y]$ is the set
of distance-decreasing maps $X \to Y$ with metric $d_\infty(f, g) = \sup_{x
\in X} d(f(x), g(x))$. (Its underlying topology is that of uniform
convergence.) Moreover, this symmetric monoidal closed structure restricts
to one on $\CptMet$.

The hypotheses of Lemma~\ref{lemma:closed} are easily verified, so
axiom~\axIs\ holds in $\CptMet$. For future use, we also note that the
monoidal closed structure gives a composition map
\begin{align}
\label{eq:comp-met}
[Y, Z] \otimes [X, Y] \to [X, Z]
\end{align}
in $\CptMet$ for each $X$, $Y$ and $Z$. In particular, composition is
continuous with respect to the product topology on the domain
of~\eqref{eq:comp-met}. 

\begin{lemma}
The isometries and surjections in $\CptMet$ define a factorization system
of finite type. 
\end{lemma}

\begin{proof}
We have already proved axioms~\axI\ and~\axIs. For axiom~\axII, a diagram
\[
\cdots \emb X_1 \emb X_0
\]
in $\CptMet$ is essentially a nested sequence of closed subspaces $X_n$ of
$X_0$, and the limit is $\bigcap_{n \in \N} X_n$.

For axiom~\axIII, consider a map of diagrams
\[
\xymatrix{ 
\cdots  \ar@{ >->}[r]                           &
X_1 \ar@{ >->}[r] \ar@{->>}[d]^{u_1}      &
X_0 \ar@{->>}[d]^{u_0}  \\
\cdots  \ar@{ >->}[r]    &
Y_1 \ar@{ >->}[r]     &
Y_0 
}
\]
in $\CptMet$. Regarding $X_n$ and $Y_n$ as subspaces of $X_0$ and $Y_0$
respectively, the induced map $u \from \bigcap X_n \to \bigcap Y_n$ on
limits is the restriction of $u_0$. To show that $u$ is surjective, let $y
\in \bigcap Y_n$. For each $n \in \N$, choose $x_n \in u_n^{-1}(y)$, as
we may since $u_n$ is surjective. The sequence $(x_n)$ in $X_0$ has a
subsequence converging to $x$, say, and then $x \in \bigcap X_n$ with $u(x)
= y$. 

For axiom~\axIIs, consider a diagram
\[
X_0 \covby{f_1} X_1 \covby{f_2} \cdots
\]
in $\CptMet$, and write $f^n = f_n \of \cdots \of f_1$. The colimit of the
diagram is $X_0/{\sim}$, where $x \sim x'$ if $\inf_n d(f^n(x), f^n(x')) =
0$. The metric on $X_0/{\sim}$ is given by 
\begin{align}
\label{eq:qt-met}
d([x], [x']) = \inf_n d(f^n(x), f^n(x')), 
\end{align}
where $x, x' \in X_0$ and $[\,\cdot\,]$ denotes equivalence class. The
coprojection $X_n \to X_0/{\sim}$ is determined by $f^n(x) \mapsto [x]$ for
all $x \in X_0$.

Finally, for axiom~\axIIIs, consider a map of diagrams
\[
\xymatrix{ 
Y_0 \ar@{->>}[r]^{g_1} \ar@{ >->}[d]^{u_0}           &
Y_1 \ar@{->>}[r]^{g_2} \ar@{ >->}[d]^{u_1}           &
\cdots                  \\
X_0 \ar@{->>}[r]_{f_1}        &
X_1 \ar@{->>}[r]_{f_2}        &
\cdots.
}
\]
Write $u\from Y_0/{\sim} \to X_0/{\sim}$ for the induced map on colimits,
given on $y \in Y_0$ by $u([y]) = [u_0(y)]$. Then for all $y, y' \in Y_0$,
\begin{align*}
d(u([y]), u([y']))      &
=
d([u_0(y)], [u_0(y')])  \\
&
=
\inf_n d(f^n u_0(y), f^n u_0(y'))       \\
&
=
\inf_n d(u_n g^n(y), u_n g^n(y'))       \\
&
=
\inf_n d(g^n(y), g^n(y'))       \\
&
=
d([y], [y']),
\end{align*}
as required.
\end{proof}

It follows from Theorem~\ref{thm:main} that every endomorphism of a compact
metric space has eventual image duality.

\begin{example}
\label{eg:not-abs}
Here we show that eventual image duality is not absolute. That is, we
construct a functor $F\from \C \to \cat{D}$ and an endomorphism
$\ento{X}{f}$ in $\C$ such that $f$ has eventual image duality but $F(f)$
does not. 

Define $\ento{[0, 1]}{f}$ in $\CptMet$ by $f(x) = x/2$. Let $\Met$ be the
category of all metric spaces and distance-decreasing maps.  Define $C\from
\CptMet^\op \to \Met$ on objects by taking $C(X)$ to be the set of
continuous functions $X \to \R$ with the sup metric, and on maps by
composition. Certainly $f$ has eventual image duality, but we show that the
canonical map from the limit of the diagram 
\begin{align}
\label{eq:not-abs}
\cdots \toby{C(f)} C[0, 1] \toby{C(f)} C[0, 1] \toby{C(f)} \cdots
\end{align}
to its colimit is not injective, so that $C(f)$ does not have eventual image
duality.

For $t \in \R$, let $\theta_t \in C[0, 1]$ denote the function $x \mapsto
tx$. Then $(\theta_{2^{-n}})_{n \in \Z}$ and $(\theta_0)_{n \in \Z}$ are
both elements of the limit of~\eqref{eq:not-abs}. Their $0$th projections
are $\theta_1, \theta_0 \in C[0, 1]$. On the other hand, it is
straightforward to show that two elements $\phi, \psi \in C[0, 1]$
represent the same element of the colimit if and only if $\phi(0) =
\psi(0)$. Since $\theta_1(0) = 0 = \theta_0(0)$, the two elements
$(\theta_{2^{-n}})$, $(\theta_0)$ of the limit map to the same element of
the colimit.
\end{example}

For the rest of this section, let $\ento{X}{f}$ be an endomorphism in
$\CptMet$. 

Theorem~\ref{thm:main}\bref{part:main-lim} shows that $\ei{f} = \bigcap_{n
\in \N} \im(f^n)$ and that $\iota_f \from \ei{f} \emb X$ is the
inclusion. On the other hand, Theorem~\ref{thm:main}\bref{part:main-colim}
shows that $\ei{f} = X/{\sim}$, where
\[
x \sim x'
\iff
\inf_{n \in \N} d(f^n(x), f^n(x')) = 0
\]
and the metric on $X/{\sim}$ is defined as in~\eqref{eq:qt-met}, and that
$\pi_f \from X \cov \ei{f}$ is the quotient map $X \to X/{\sim}$.

By Theorem~\ref{thm:main}, the map
\[
\pi_f \iota_f \from \bigcap_{n \in \N} \im(f^n) \to X/{\sim}
\]
is an isomorphism. For $\pi_f \iota_f$ to be a bijection means that for
all $x \in X$, there is a unique $y \in \bigcap \im(f^n)$ such that $\inf_n
d(f^n(x), f^n(y)) = 0$. Identifying $\bigcap \im(f^n)$ with $X/{\sim}$ via
$\pi_f \iota_f$, this element $y$ is $\iota_f\pi_f(x) =
f^\infty(x)$. Thus, $f^\infty(x)$ is the unique element of $\bigcap
\im(f^n)$ satisfying
\[
\inf_{n \in \N} d\bigl(f^n(x), f^n(f^\infty(x))\bigr) = 0,
\]
or equivalently
\begin{align}
\label{eq:met-idem-lim}
\lim_{n \to \infty} d\bigl(f^n(x), f^n(f^\infty(x))\bigr) = 0.
\end{align}

The back and forth description of $f^\infty(x)$ for sets and vector spaces
(Propositions~\ref{propn:set-trace} and~\ref{propn:vect-trace}) has the
following metric analogue.

\begin{propn}
\label{propn:met-trace}
Let $x \in X$. Let $(f^{n_i}(x))$ be a convergent subsequence of
$(f^n(x))$, with limit $z$. Then $z \in \ei{f}$ and
$f^\infty(x) = \lim\limits_{i \to \infty} \au{f}^{-n_i}(z)$. 
\end{propn}

\begin{proof}
First, $z \in \bigcap_{i \in \N} \im(f^{n_i}) = \ei{f}$. Now for all $i \in
\N$, 
\begin{align*}
d\bigl(\au{f}^{-n_i}(z), f^\infty(x)\bigr)      &
=
d\bigl(z, f^{n_i}(f^\infty(x))\bigr)    \\
&
\leq
d\bigl(z, f^{n_i}(x)\bigr) + d\bigl(f^{n_i}(x), f^{n_i}(f^\infty(x))\bigr).
\end{align*}
From the fact that $z = \lim_{i \to \infty} f^{n_i}(x)$ and
equation~\eqref{eq:met-idem-lim}, it follows that $\au{f}^{-n_i}(z) \to
f^\infty(x)$ as $i \to \infty$. 
\end{proof}

In Proposition~\ref{propn:met-trace}, $f^\infty$ is constructed as a
\emph{pointwise} limit. We now construct $f^\infty$ as a \emph{uniform}
limit. Recall that our function spaces $[X, Y]$ have the topology of
uniform convergence.

\begin{propn}
\label{propn:met-trace-ufm}
Let $(f^{n_i})$ be a convergent subsequence of $(f^n)$ in $[X, X]$, with
limit $h$. Then $\im(h) \sub \ei{f}$ and, writing $h = \bigl( X \toby{h'}
\ei{f} \embby{\iota_f} X \bigr)$, we have $f^\infty = \lim\limits_{i \to
\infty} \iota_f \of \au{f}^{-n_i} \of h'$ in $[X, X]$.
\end{propn}

Since $[X, X]$ is compact, $(f^n)$ does have a convergent subsequence. 

\begin{proof}
By Proposition~\ref{propn:met-trace}, $\im(h) \sub \ei{f}$
and $\iota_f \of \au{f}^{-n_i}\of h'$ converges to $f^\infty$
pointwise. So it suffices to show that $\iota_f \of \au{f}^{-n_i} \of h'$
converges uniformly, that is, in $[X, X]$. Since $[X, X]$ is compact, it is
complete, so we need only show that this sequence is Cauchy. And indeed,
for all $i \geq j \geq 0$,
\begin{align*}
d_\infty\bigl(\iota_f \of \au{f}^{-n_i} \of h', 
\iota_f \of \au{f}^{-n_j} \of h'\bigr)       &
=
d_\infty\bigl(\au{f}^{-n_i} \of h', \au{f}^{-n_j} \of h'\bigr)    \\
&
=
d_\infty\bigl(h', \au{f}^{n_i - n_j} \of h'\bigr) \\
&
=
d_\infty(h, f^{n_i - n_j} \of h)        \\
&
\leq
d_\infty(h, f^{n_i}) 
+ d_\infty(f^{n_i - n_j} \of f^{n_j}, f^{n_i - n_j} \of h)      \\
&
\leq
d_\infty(h, f^{n_i}) + d_\infty(f^{n_j}, h),
\end{align*}
giving the result. 
\end{proof}

Lemma~\ref{lemma:ch} implies that the inverse of a linear automorphism $g$
belongs to $\spn\{1, g, g^2, \ldots\}$. The metric analogue is as follows.

\begin{lemma}
\label{lemma:met-ch}
Let $g$ be an automorphism of a compact metric space $Y$. Then $g^{-1} \in
\Cl\{1_Y, g, g^2, \ldots\}$, where $\Cl$ is the closure operator on $[Y, Y]$.
\end{lemma}

\begin{proof}
Write $\metspn{g} = \Cl\{1_Y, g, g^2, \ldots\}$. Since $[Y, Y]$ is compact,
so is $\metspn{g}$. The automorphism $g$ of $Y$ induces an automorphism
$g \of -$ of $[Y, Y]$, which restricts to an endomorphism of
$\metspn{g}$. Then $g \of -$ is a self-isometry of the compact metric space
$\metspn{g}$, so $g \of -$ is an automorphism of $\metspn{g}$. Since $1_Y \in
\metspn{g}$, there is some $g' \in \metspn{g}$ such that $g \of g' =
1_Y$. But $g$ is invertible, so $g^{-1} = g' \in \metspn{g}$.
\end{proof}

\begin{propn}
\label{propn:met-inv-cl}
$\au{f}^{-1} \in \Cl\{1_{\ei{f}}, \au{f}, \au{f}^2, \ldots\}$, where $\Cl$
is the closure operator on $[\ei{f}, \ei{f}]$.
\qed
\end{propn}

In the next two results, we use Proposition~\ref{propn:met-inv-cl} to give
a further characterization of $f^\infty$. The first is a variant of
Lemma~1.3 of Borges~\cite{BorgHRH}.

\begin{lemma}
\label{lemma:met-idem-cl}
The idempotent $f^\infty$ on $X$ belongs to $\Cl\{f, f^2, \ldots\}$.
\end{lemma}

\begin{proof}
Since $[X, X]$ is compact, the sequence $(f^n)_{n \geq 1}$ has a convergent
subsequence $(f^{n_i})_{i \geq 0}$. Write $h$ for its limit. By
Proposition~\ref{propn:met-trace-ufm}, $\im(h) \sub \ei{f}$ and, writing $h
= \bigl( X \toby{h'} \ei{f} \embby{\iota_f} X\bigr)$, we have
\begin{align}
\label{eq:idem-met-uc}
f^\infty = \lim_{i \to \infty} \iota_f \of \au{f}^{-n_i} \of h'
\end{align}
in $[X, X]$. 

We now repeatedly use the continuity of composition, shown
in~\eqref{eq:comp-met}. By Proposition~\ref{propn:met-inv-cl}, $\au{f}^{-1}
\in \Cl\{\au{f}^{\,k} \such k \geq 0\}$. Now $\{\au{f}^{\,k} \such k \geq
0\}$ is closed under composition, so its closure is too, giving 
$\au{f}^{-n_i} \in \Cl\{\au{f}^{\,k} \such k \geq 0\}$
for each $i$. It follows that for each $i$,
%
\[
\iota_f \of \au{f}^{-n_i} \of h'        
\in 
\Cl\{ \iota_f \of \au{f}^{\,k} \of h' \such k \geq 0 \}     
=
\Cl\{ f^k \of h \such k \geq 0 \}.
\]
%
But $h \in \Cl \{ f^n \such n \geq 1\}$ by definition of $h$, so for each
$k \geq 0$,
\[
f^k \of h 
\in 
\Cl\{f^{k + n} \such n \geq 1\} 
\sub 
\Cl\{f^n \such n \geq 1\}.
\]
Hence for each $i$,
\[
\iota_f \of \au{f}^{-n_i} \of h' \in \Cl\{f^n \such n \geq 1\},
\]
and the result follows from equation~\eqref{eq:idem-met-uc}.
\end{proof}

\begin{remark}
Lemma~\ref{lemma:met-idem-cl} and the idempotence of $f^\infty$ imply that
$f^\infty \in \Cl\{f^r, f^{r + 1}, \ldots\}$ for all $r \geq 0$. Analogous
results hold for sets and vector spaces. 
\end{remark}

\begin{propn}
$f^\infty$ is the unique idempotent element of $\Cl\{f, f^2, \ldots\}$.
\end{propn}

This is the metric analogue of Proposition~\ref{propn:sets-unique-idem} for
sets.

\begin{proof}
Let $e$ be an idempotent in $\Cl\{f, f^2, \ldots\}$. By
Lemma~\ref{lemma:met-idem-cl}, it suffices to prove that $e = f^\infty$.
We will show that the idempotents $e$ and $f^\infty$ commute and have the
same image. It will follow that $e = f^\infty$: for since $\im(e) \sub
\im(f^\infty)$ and $f^\infty$ is idempotent, $f^\infty e = e$, and
similarly $e f^\infty = f^\infty$, giving the result.

First, $e$ and $f^\infty$ commute. Indeed, the composition map $[X, X]
\times [X, X] \to [X, X]$ is continuous and restricts to a commutative
operation on $\{1, f, f^2, \ldots\}$, so it also restricts to a
commutative operation on $\Cl\{1, f, f^2, \ldots\}$, which contains
both $e$ and $f^\infty$ (by Lemma~\ref{lemma:met-idem-cl}).

Next, $\im(e) \sub \im(f^\infty)$. For let $n \geq 1$. The endomorphism of
$[X, X]$ defined by $h \mapsto h^n$ restricts to a map $\{f, f^2, \ldots\}
\to \{f^n, f^{n + 1}, \ldots\}$, and is continuous, so it also restricts to
a map
\[
\Cl\{f, f^2, \ldots\} \to \Cl\{f^n, f^{n + 1}, \ldots\}.
\]
Hence $e = e^n \in \Cl\{f^n, f^{n + 1}, \ldots\}$, and it follows that
$\im(e) \sub \im(f^n)$. This holds for all $n$, so $\im(e) \sub \bigcap
\im(f^n) = \ei{f} = \im(f^\infty)$.

Finally, $\im(f^\infty) \sub \im(e)$. Indeed, $f^n
(\ei{f}) \sub \ei{f}$ for each $n$, so $e$ restricts to an endomorphism
$\hat{e}$ of $\ei{f}$. But $f^n|_{\ei{f}}$ is an isometry for each $n$, so
$\hat{e}$ is also an isometry. Hence $\hat{e}$ is a self-isometry
of the compact space $\ei{f}$, and therefore surjective. It follows that
$\im(e) \supseteq \ei{f} = \im(f^\infty)$. 
\end{proof}

Thus, $\ei{f}$ can be characterized as the image or fixed set of the
unique idempotent in $\Cl\{f, f^2, \ldots\}$.  

Theorem~\ref{thm:coalg}
provides a further characterization: $\ei{f}$ is the largest closed
subspace $V$ of $X$ such that $V \sub fV$.

There is yet another characterization. A point $x \in X$ is 
\demph{recurrent} for $f$ if $x \in \Cl\{f(x), f^2(x), \ldots\}$
(\cite{HNV}, Section~h-6.6).

\begin{propn}
The set of recurrent points for $f$ is $\ei{f}$. In particular, the set of
recurrent points is closed.
\end{propn}

\begin{proof}
Let $x$ be a recurrent point. We prove by induction that $x \in
\im(f^n)$ for all $n \in \N$, which will imply that $x \in \ei{f}$. The
base case is trivial. For $n \geq 0$, assume inductively that $x \in
\im(f^n)$. Then $f(x), f^2(x), \ldots$ all belong to the closed set
$\im(f^{n + 1})$, which therefore also contains $x$, completing the
induction. The converse follows from Lemma~\ref{lemma:met-idem-cl}. 
\end{proof}

Finally, note that the endomorphism $f \mapsto f^\infty$ of $[X, X]$ is
typically discontinuous. For example, let $X = [0, 1]$, and for $0 \leq t
\leq 1$, define $f_t \from X \to X$ by $f_t(x) = tx$. Then $f_t^\infty$ has
constant value $0$ whenever $t < 1$, but $f_1^\infty$ is the
identity. Thus, as $t \to 1^-$, we have $f_t \to f_1$ but $f_t^\infty
\not\to f_1^\infty$. Even in $\CptMet$, long-term dynamics are sensitive to
small changes in parameters.

\section{Further examples}
\label{sec:other}

We end with four further examples of categories with eventual image
duality: functor categories where the codomain has eventual image duality,
categories of finite models for a finitary algebraic theory, the category
of finite partially ordered sets, and Cauchy-complete categories with
finite hom-sets.

\begin{propn}
Let $\scat{A}$ be a small category and $\C$ a category with eventual
image duality. Then the functor category $\C^{\scat{A}}$ has eventual
image duality, and eventual images in it are computed pointwise.
\end{propn}

\begin{proof}
This follows from the fact that limits and colimits in functor categories
are computed pointwise (Kelly~\cite{KellBCE}, Section~3.3).
\end{proof}

\begin{example}
Let $G$ be a group, let $X$ be a finite-dimensional representation of $G$,
and let $f$ be a $G$-equivariant endomorphism of $X$. Then the eventual
image of $f$, as an endomorphism in the category of representations of $G$,
is the eventual image in $\FDVect$ equipped with the natural $G$-action.
\end{example}

If a finite set $X$ has the structure of a group, ring, etc., and if an
endomorphism $f$ of $X$ preserves that structure, then $\ei{f}$ is
naturally a group, ring, etc., and the maps $\iota_f$, $\pi_f$ and
$f^\infty$ are homomorphisms. In general:

\begin{propn}
Let $T$ be a finitary algebraic theory. Write $\C$ for the category of
$T$-algebras with finite underlying set. Then $\C$ has eventual image
duality, and eventual images in it are computed as in $\FinSet$. 
\end{propn}

\begin{proof}
The injections and surjections form a factorization system on
$\C$, which we show to be of finite type. Any injective or surjective
endomorphism of a finite $T$-algebra is bijective and so invertible,
giving axioms~\axI\ and~\axIs. The forgetful functor from the
category of $T$-algebras to $\Set$ creates limits and filtered
colimits, so the other axioms follow. The result then follows from
Theorem~\ref{thm:main}.
\end{proof}

The theory of partially ordered sets is not algebraic. Nevertheless:

\begin{propn}
The category of finite partially ordered sets has eventual image duality.
\end{propn}

\begin{proof}
In the category $\C$ of finite posets, the injections and surjections form
a factorization system. We show that it is of finite type. For axioms \axI\ and
\axIs, let $\ento{X}{f}$ be an injection or surjection in
$\C$. Then $f$ is a bijection. Since $X$ is finite, $f^n = 1_X$ for some $n
\geq 1$; then the set-theoretic inverse $f^{-1}$ is $f^{n - 1}$, which is
order-preserving, so $f$ is an order-isomorphism. The other axioms follow
from the fact that sequential limits and colimits in $\C$ are computed as
in $\Set$. Hence Theorem~\ref{thm:main} applies.
\end{proof}

The last two propositions also follow from our final result. 

\begin{thm}
Let $\C$ be a Cauchy-complete category in which every hom-set is
finite. Then $\C$ has eventual image duality. 
\end{thm}

\begin{proof}
Let $\ento{X}{f}$ be an endomorphism in $\C$. By
Lemma~\ref{lemma:semi-idem} applied to the finite semigroup $\{f, f^2,
\ldots\}$, we can choose $N \geq 1$ such that $f^N$ is idempotent. Since
$\C$ is Cauchy-complete, $f^N$ has a splitting
\[
\xymatrix{I \ar@<0.5ex>[r]^i & X. \ar@<0.5ex>[l]^p}
\]
We will show that $I$ is an eventual image of $f$, with $i$ and $p$ as the
maps usually called $\iota_f $ and $\pi_f$. (\emph{If} $f$ has
eventual image duality then $\ei{f}$ must be $I$: for $\ei{f} = \ei{f^N} =
I$ by Proposition~\ref{propn:timescale} and Example~\ref{eg:ei-idem}.)

Put $\hat{f} = pfi \from I \to I$. Then the diagram
\begin{align}
\label{eq:cc-fact}
\begin{array}{c}
\xymatrix{
X \ar[r]^p \ar[d]_f     &
I \ar[r]^i \ar[d]^{\hat{f}}     &
X \ar[d]^f      \\
X \ar[r]_p      &
I \ar[r]_i      &
X
}
\end{array}
\end{align}
commutes, since 
\[
\hat{f}p = pfip = pf f^N = p f^N f = pipf = pf
\]
and dually for the right-hand square. So
\[
\hat{f}^N = \hat{f}^N pi = p f^N i = pipi = 1_X.
\]
Hence $\hat{f}$ is an automorphism of $I$. 

By~\eqref{eq:cc-fact}, there is a cone $\bigl( I \toby{\pr_n} X \bigr)_{n
\in \Z}$ on $\ \cdots \toby{f} X \toby{f} \cdots\ $ given by
\[
\pr_n = \Bigl( I \toby{\hat{f}^n} I \toby{i} X \Bigr).
\]
We show that this cone is a limit. Let $\bigl( A \toby{q_n} X \bigr)_{n \in
\Z}$ be any cone on the same diagram. We must prove that there is a unique
map $\bar{q} \from A \to I$ such that
\begin{align}
\label{eq:cc-int}
q_n = \bigl( A \toby{\bar{q}} I \toby{\pr_n} X \bigr)
\end{align}
for all $n \in \Z$. 

For uniqueness, equation~\eqref{eq:cc-int} with $n = 0$
states that $q_0 = i \bar{q}$, giving $\bar{q} = pq_0$.  For existence,
put $\bar{q} = pq_0$. Note that for all $n \in \Z$,
\[
q_n = f^N \of q_{n - N} = f^{2N} \of q_{n - N} = q_{n + N}.
\]
Now equation~\eqref{eq:cc-int} follows from the commutativity of the
diagram
\[
\xymatrix@R+1ex@C+2ex{
A \ar[r]^{q_0} \ar@/^1.5pc/[rr]^{\bar{q}} 
\ar[rd]_{q_n} \ar@/_2pc/[rrdd]_{q_n = q_{n + N}}&
X \ar[r]^p \ar[d]^{f^n} &
I \ar[d]^{\hat{f}^N} \ar@/^2pc/[dd]^{\pr_n}   \\
&
X \ar[r]^p \ar[rd]_{f^N}        &
I \ar[d]^i      \\
&
&
X,
}
\]
where we have used the cone property of $(q_n)$ and
diagram~\eqref{eq:cc-fact}. 

This proves that $\bigl(I \toby{\pr_n} X\bigr)_{n \in \N}$ is a limit
cone. Dually, $\bigl( X \toby{\copr_n} I \bigr)_{n \in \N}$ is a colimit
cone, where $\copr_n = \hat{f}^{-n} p$. The composite 
\[
I \toby{\pr_0} X \toby{\copr_0} I
\]
is $pi = 1_I$, which is an isomorphism. Hence $f$ has eventual image
duality. 
\end{proof}

\ucontents{section}{References}
\bibliography{mathrefs}

\end{document}